\documentclass[aps, prx,reprint,superscriptaddress]{revtex4-1}

\usepackage[english]{babel}

\usepackage{graphicx} 
\usepackage{romannum} 
\usepackage{amssymb}  
\usepackage{amsmath}  
\usepackage{xcolor}   

\usepackage{amsthm} 
\newtheorem{lemma}{Lemma}
\newtheorem{theorem}{Theorem}
\newtheorem{definition}{Definition}

\renewcommand{\vec}[1]{{\mbox{\mathversion{bold}\ensuremath{#1}}}} 

\usepackage{color}

\begin{document}


\author{Dominik Kahl}
\affiliation{Department of Mathematics and Technology, 
	University of Applied Sciences Koblenz, Joseph-Rovan-Allee 2, 53424 Remagen, Germany}


\author{Philipp Wendland}
\affiliation{Department of Mathematics and Technology, 
	University of Applied Sciences Koblenz, Joseph-Rovan-Allee 2, 53424 Remagen, Germany}

\author{Matthias Neidhardt}
\affiliation{Institute for Computer Science, University of Bonn, Endenicher 
	Allee 19a, 53115 Bonn, Germany}
\author{Andreas Weber}

\affiliation{Institute for Computer Science, University of Bonn, Endenicher 
	Allee 19a, 53115 Bonn, Germany}

\author{Maik Kschischo}
\email[]{kschischo@rheinahrcampus.de}
\affiliation{Department of Mathematics and Technology, 
	University of Applied Sciences Koblenz, Joseph-Rovan-Allee 2, 53424 Remagen, Germany}

\date{\today}


\begin{abstract}
	Despite recent progress in our understanding of complex 
	dynamic networks, it remains challenging 	to devise
	sufficiently accurate models to observe, control or predict 
	the state of real systems in biology, economics or other fields. 
	A largely overlooked fact is that these systems are typically open 
	and receive unknown inputs from their environment. A further 
	fundamental obstacle are structural model errors caused by 
	insufficient or inaccurate knowledge about the quantitative interactions 
	in the real system. 	
	
	Here, we show that unknown inputs to open systems and model errors can 
	be treated	under the common framework of invertibility, which is a requirement 
	for reconstructing these disturbances from output measurements. By exploiting 
	the fact that invertibility can be decided from the influence graph of the 
	system, we analyse the relationship between structural network properties 
	and invertibility under different realistic scenarios. We show that sparsely 
	connected  	scale free networks are the most difficult to invert. We introduce a new 
	sensor node placement algorithm to select a minimum set of measurement 
	positions in the network required for invertibility. This algorithm 
	facilitates optimal experimental design for the reconstruction of inputs 
	or model errors from output measurements. Our results have both 
	fundamental and practical implications for nonlinear systems analysis, modelling 
	and design.
\end{abstract}


\title{
	Structural Invertibility and Optimal Sensor Node Placement for Error 	
	and Input Reconstruction in Dynamic Systems
	}
\maketitle


\section{Introduction}

	\begin{figure*}
		\centering
		\includegraphics[width=2\columnwidth]{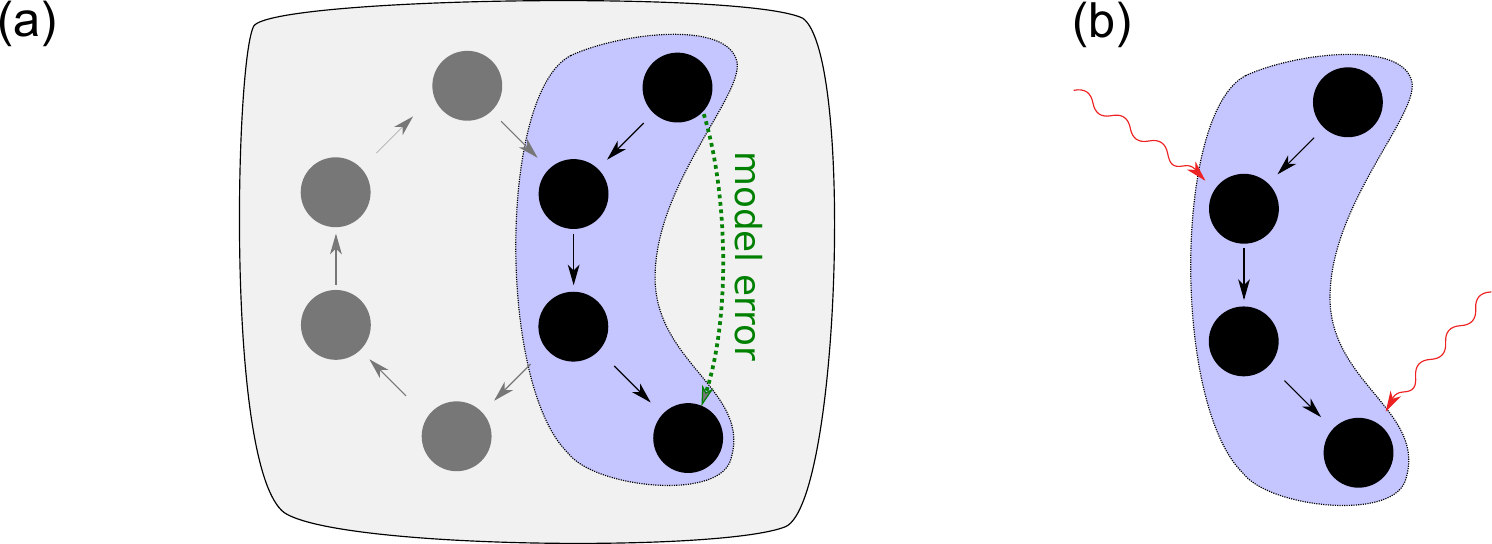}
		\caption{
			Schematic illustration of an open system and its environment. 
			(a) The open system (blue shadowed region) interacts with its 
			dynamic environment. The state (black nodes) of the open system 
			is determined by both internal interactions and by the 
			interactions with the (unknown) state of the environment 	
			(grey nodes). There are also systematic model errors, as indicated by a 
			missing interaction between state variables (green dashed line). 
			(b) Both the effects of the environment on the state of the open 
			system and model errors can be regarded as unknown inputs. 
			}
		\label{fig:Fig1}
	\end{figure*}
	\begin{figure*}
		\centering
		\includegraphics[width=1.8\columnwidth]{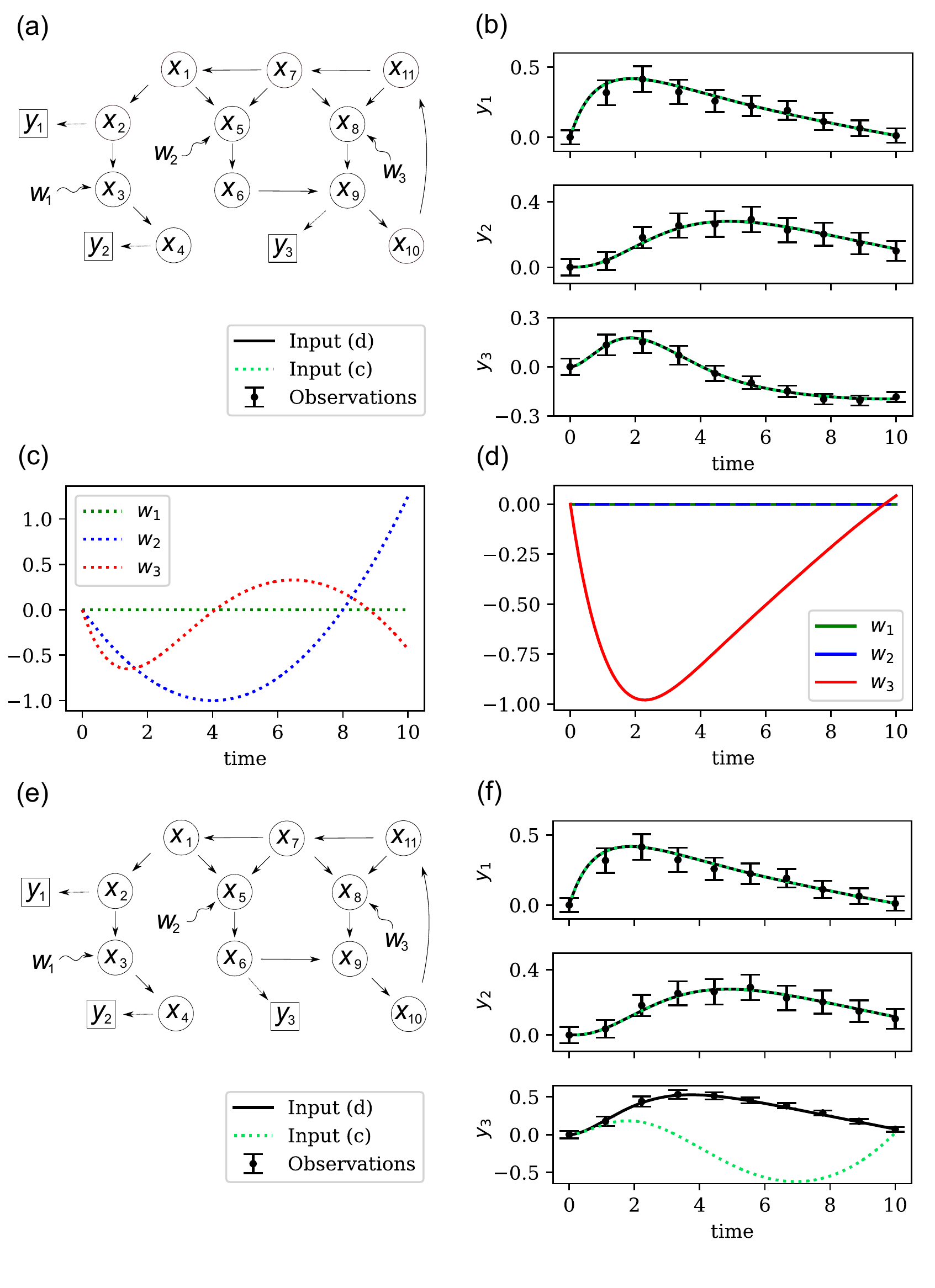}
		\caption{
			Invertibility of a dynamic system. (a)~The
			influence graph represents the states 
			$\vec{x}=(x_1,\ldots, x_9)$ 
			of a system as nodes (circles). Endogenous interactions are 
			indicated by arrows between the states. This system receives 
			three input signals 
			$\vec{w}(t)=(w_1(t),w_2(t),w_3(t))$ 
			(wiggly arrows) targeting the input nodes 
			$S=\{x_3,x_5, x_8\}$. The output 
			$\vec{y}(t)=(y_1(t), y_2(t), y_3(t))$ 
			(squares) is given by measurements of the state variables in the 
			sensor node set 
			$Z=\{x_2, x_4, x_9\}$. 
			(b)~The measured output data (dots) of the open system can be 
			explained by different input signals (c,d) causing identical 
			outputs, as indicated by the dotted line for input (c) and  by 
			the solid line for input (d). Thus, the system is not invertible. 
			(e) For 
			$Z=\{x_2, x_4, x_6\}$, i.e. if the sensor for the output 
			$y_3$ 
			is moved to state 
			$x_6$ 
			instead of 
			$x_9$ 
			(compare (a)), we obtain an invertible system and the input 
			signal in (d) is unique for the observed output in (f).}
		\label{fig:Fig2}
		\end{figure*}

	Dynamic systems in such diverse areas like physics, biology, 
	economics or engineering are often composed of 
	many interacting components. Developing useful and sufficiently  
	accurate models of such complex dynamic networks having many degrees of freedom
	remains challenging
	\cite{	raue_lessons_2013,
		 	almog_is_2016,
		 	janes_engineering_2017,
		 	tsigkinopoulou_respectful_2017} 
	despite the ever increasing size of network data sets providing the 
	wiring diagrams of diverse systems
	\cite{	nr, 
			batagelj_pajek_nodate, 
			kunegis_konect_2013,
			leskovec_snap_2014}.

	One important challenge for modelling real dynamic networks is that they 
	are open systems receiving inputs from their environment, see 
	Fig.~\ref{fig:Fig1}. These inputs need to be either known or under 
	experimental control to fully characterise the dynamic state of the 
	network. For example, a biological cell is a system with a 
	certain autonomy, but at the same time is crucially dependent on signals 
	and nutrients received from the exterior. It is practically impossible to 
	simultaneously detect or control all signals received by a 
	living cell in their natural environment and to measure all compounds 
	exchanged with the extracellular space. As another example, consider a 
	population dynamic system in a certain geographical area. The 
	state, i.e. the population count of the different species in this area, 
	is not only determined by the inner dynamic interactions (e.g.~pray 
	and predator relationships) between the species, but also
	by migration and by environmental factors. Again, for state estimation it 
	is typically neither feasible to directly quantify
	all these inputs nor is it possible to ignore them. The same applies to 
	physical, engineering, or economic systems, which will always be subject 
	to inputs and disturbances from the environment.
	
	If the inputs to the system cannot directly be obtained, then we might 
	want to infer the inputs from the 
	measurable outputs. For the biological cell, we
	can try to estimate the transport fluxes across the membrane and the 
	signals received by the cell from 
	time series of measured protein concentrations. 
	For a population, the number of certain species will be monitored and we 
	will try to estimate dynamic 
	changes of birth and migration rates for other, not directly observed species.  Algorithms to 
	estimate the inputs from the 
	outputs of systems described by ordinary differential equations (ODEs) 
	are an ongoing research topic, see e.g. 
	\cite{  kuhl_real-time_2011,
			schelker_comprehensive_2012,
			Fonod_boblin_unknown_2014, 
			engelhardt_learning_2016, 
			engelhardt_bayesian_2017, 
			chakrabarty_state_2017, 
			tsiantis_optimality_2018}.
	However, no such algorithm can succeed, if the output doesn't provide 
	sufficient information about the input. Mathematically, this means that 
	the map from input to output is not \textit{invertible} and thus, systems 
	inversion is bound to fail. 

	The situation is illustrated in Fig.~2(a), where a 
	hypothetical system is represented by an influence graph. The nodes 
	correspond to the systems states 
	$\vec{x}=(x_1,\ldots x_{11})$
	and the black arrows indicate the endogenous interactions amongst them. 
	The open system receives three input signals 
	$\vec{w}(t)=(w_1(t), w_2(t),w_3(t))$ 
	(wiggly arrows) from its exosystem, targeting the set 
	$S=\{x_3, x_5, x_8\}$
	of three input nodes. The output signal 
	$\vec{y}(t)=(y_1(t),y_2(t),y_3(t))$ 
	in Fig.~2(b) is formed by measuring the time course of the 
	sensor node set 
	$Z=\{x_2, x_4, x_9\}$, i.e.$y_1(t)=x_2(t)$, $y_2(t)=x_4(t)$ 
	and $y_3(t)=x_9(t)$. 
	Nonetheless, these output observations are insufficient to uniquely 
	reconstruct the corresponding input signals. Indeed, the two different 
	input signals~(Fig.~\ref{fig:Fig2}(c,d)) generate identical outputs 
	(solid and dahed lines Fig.~\ref{fig:Fig2}(b)) and both reproduce the data 
	points with good accuracy. Thus, the system is not invertible. 	
	
	The lack of invertibility can be remedied by a careful experimental 
	design. In the example of Fig.~\ref{fig:Fig2}, the system can be made 
	invertible by measuring the state $y_3(t)=x_6(t)$ instead of 
	$x_9(t)$, compare Fig.~\ref{fig:Fig2}(a) and (e). As explained below, 
	the modified output from the sensor node set $Z=\{x_2, x_4, x_6\}$ 
	provides sufficient information to uniquely identify the input signal 
	in~Fig.~\ref{fig:Fig2}(c) as the cause for the observed 
	data~(Fig.~\ref{fig:Fig2}(f)). This example highlights the need for a \textit{sensor 
	node placement algorithm for invertibility} of open systems, which is one 
	important result of this text. 
	
	Systematic model errors are another important source of potential discrepancies between
	measurements and model outputs. One type of model errors is caused by an incorrect influence
	graph, i.e. by missing or spurious interactions between the state variables (Fig.~\ref{fig:Fig1}). Another type
	of model errors originates from misspecifications of the functional form or parameters of the 
	interactions. However, both types of these endogenous model errors can effectively be treated as
	unknown inputs to the model system at hand, see Sec.~\ref{sec:open}. As for genuine exogenous inputs, a unique reconstruction of these unknown 
	inputs caused by endogenous model errors is again only possible if the systems model is invertible. 

	Controllability and observability are other important systems properties, which have attracted
	renewed interest from a complex systems point of view
		\cite{
		liu_controllability_2011, 
		cornelius_controlling_2011, 
		liu_observability_2013, 
		sun_controllability_2013, 
		cornelius_realistic_2013, 
		gao_target_2014,
		yan_spectrum_2015, 
		motter_networkcontrology_2015, 
		liu_control_2016,
		summers_submodularity_2016, 
		aguirre_observ_2018,
		haber_state_2018}.
	Structural approaches use only the influence graph
	\cite{
		lin_structural_1974, 
		liu_controllability_2011, 
		liu_observability_2013, 
		gao_target_2014, 
		motter_networkcontrology_2015, 
		liu_control_2016}
	and are therefore well suited to examine the controllability/observability 
	properties related to topological network features. Structural controllability
	(observability) analysis provide binary decisions whether a system is controllable 
	(observable) or not. Later, it was emphasised that realistic control and state estimation
	often require quantitative information about the network interactions 
    and parameters of the system~
	\cite{ 	krener_meas_2009,
			cornelius_controlling_2011, 
			sun_controllability_2013, 
			cornelius_realistic_2013, 
			yan_spectrum_2015,
			lo_iudice_structural_2015,
			summers_submodularity_2016, 
			klickstein_energy_2017,
			aguirre_observ_2018,
			haber_state_2018}.
	The underlying reason is that the specific functional form of the couplings 
    might require huge energies for control or very sensitive measurements
    for state estimation. Continuous measures quantifying the
	degree of controllability or observability 
	\cite{ 	krener_meas_2009,
			cornelius_controlling_2011, 
			sun_controllability_2013, 
			cornelius_realistic_2013, 
			yan_spectrum_2015,
			lo_iudice_structural_2015,
			summers_submodularity_2016, 
			klickstein_energy_2017,
			aguirre_observ_2018,
			haber_state_2018}
    were derived as alternatives to binary decisions about controllability or 
    observability.

	In contrast, the invertibility of open systems has not sufficiently been 
	investigated in the context of large complex dynamic systems. Work in 
	controllabity and observability of 
	complex networks builds on older results from the control engineering 
	literature 
	\cite{
		lin_structural_1974},
	providing algebraic and graphical conditions for both 
	traits.~The situation is similar for invertibility: Theorems providing algebraic 
	conditions for the invertibility of linear systems were already published 
	in the 1960s 
	\cite{
		brockett_reproducibility_1965, 
		silverman_inversion_1969, 
		sain_invertibility_1969}
	and later extended to nonlinear systems
	\citep{
		hirschorn_invertibility_1979, 
		nijmeijer_invertibility_1982, 
		nijmeijer_right-invertibility_1986, 
		fliess_note_1986}.
	These algebraic conditions require a full specification of the ODE 
	system including all the parameters. Even in the rather exceptional case, 
	where these data are available, the numerical test of these conditions is 
	computationally very demanding for large networks. Fortunately, invertibility 
	is a structural property~
	\cite{
		wey_rank_1998, 
		dion_generic_2003},
	which can for all practical purposes be inferred
	from the influence graph and the input and output node sets $S$ and 
	$Z$ (see Fig.~2).	
	This {\em structural invertibility} condition (see~Sec.~\ref{sec:criteria_inv})
	can efficiently be tested even for very large linear or nonlinear systems with thousands
	of state nodes. 
	
	In this text, we add invertibility as a systems property which is essential for 
	understanding open and incompletely known systems. As our main new contributions we
	\begin{itemize}
	\item show that unknown inputs 
	to open systems and structural model errors can be treated under the
	the same conceptual framework~(Sec.~\ref{sec:open});
	\item establish invertibility as a necessary condition for unknown input observers and 
	input reconstruction algorithms to work~(Subsec.~\ref{ssec:practinv});
	\item provide a new recursive algebraic criterion for invertibility~(Subsec.~\ref{subsec:iter_crit});
	\item discover important structural network properties influencing invertibility~(Sec.~\ref{sec:net});
	\item provide a simple but efficient algorithm for sensor node placement to achieve structural invertibility 
	with a minimum number of measured outputs~(Sec.~\ref{sec:sensor}). 
	\end{itemize}

	First we show, that structural model errors in nonlinear 
	dynamic system can be treated as unknown  inputs~(Sec.~\ref{sec:open}). 
	Thus, the question of whether it is possible to reconstruct model errors and unknown 
	inputs to open systems can be treated under the common framework of invertibility. Second, 
	we provide a novel criterion for the invertibility of linear systems, which can be used 
	for medium sized networks up to a few hundred state nodes due to its recursive nature~(Sec.~	
	\ref{sec:criteria_inv}). 
	For large systems we exploit the structural invertibility criterion, which 
	uses only the graph structure encoded in the interactions of the system.	
	We will also briefly touch upon the topic of practical systems inversion, which 
	requires careful regularisation schemes even for invertible systems.
	Throughout this text, we will 	consider three different scenarios (SC~\Romannum{1}-\Romannum{3}).  
	Scenario SC~\Romannum{1} refers to the case, that the positions of the inputs and outputs 
	of the system are given and that we have no opportunity to deliberately choose 
	these positions (Sec.~\ref{sec:net}). For SC~\Romannum{1} we show that 
	invertibility depends largely on the degree distribution of the influence graph 
	and that many sparse and scale free networks tend to be non-invertible. Since 
	many real networks show this characteristics, we assume in scenario 
	SC~\Romannum{2}, that the inputs are given, but the output 
	positions can be chosen. As an important result we present a sensor node 
	placement algorithm in Sec.~\ref{sec:sensor}, which provides a minimum set 
	of outputs required to uniquely reconstruct the inputs. This algorithm is 
	also useful in scenario SC~\Romannum{3}, where we assume that the positions of 
	both inputs and outputs can be manipulated. We show that placing the inputs at
	hubs with a high out-degree can drastically increase the probability that a 
	certain dynamic network can be made invertible, if in addition the outputs
	are suitably chosen, e.g. by our sensor node placement algorithm. In~Sec.~\ref{sec:conclusion} 
	we discuss the far reaching implications of invertibility for nonlinear systems 
	analysis and some open questions for further research.

\section{Open systems, unknown inputs and model errors}\label{sec:open}

	There are  two basic reasons for discrepancies between observed time 
	series data from a real world open system and the output of a 
	mathematical model for the system: First, the system might receive 
	unknown inputs
	\cite{
		mook_minimum_1987,
		boukhobza_state_2007, 
		moreno_dynamical_2014,
		engelhardt_learning_2016,
		chakrabarty_state_2017, 
		engelhardt_bayesian_2017}
	from the environment, which are not covered by the model, but nevertheless
	influence the state and the resulting measured output. 
	Second, there might be model errors, i.e. misspecified  functional 
	descriptions or missing interactions between internal model states or 
	inaccurate parameter values. In this section, we define open systems with 
	unknown inputs and then show that for ODE models, both types of error 
	can be treated as additive unknown inputs to the model. 

	\subsection{Open systems with unknown inputs}
		Consider a dynamic system 
		$\mathcal{S}_{o}$
		with time dependent state 	
		vector $\vec{x}(t) = (x_1(t),\ldots,x_N(t))^T\in \mathcal{X}$.
		The state space $\mathcal{X}$ is either $\mathbb{R}^N$, 
		a subset of it, or an $N$-manifold. 
		A dynamic model of the open system can be formulated as a 
		system of ordinary differential equations~(ODEs)
		\begin{subequations}\label{eq:ss_model_affine}
			\begin{eqnarray}
			\dot{\vec{x}}(t) &=& \vec{f}\left(\vec{x}(t)\right) + D \vec{w}(t)\\
			\vec{x}(0) &=& \vec{x}_0\\
			\vec{y}(t) &=& \vec{c} \left(\vec{x}(t)\right) \, .
			\end{eqnarray}
		\end{subequations}
		The vector field 
		$\vec{f}(\vec{x})=(f_1(\vec{x}),\ldots, f_N(\vec{x}))^T$ 
		represents the internal dynamic interactions between the state 
		variables. In addition, the system receives $M$ unknown inputs, 
		which are collected in the vector function 
		\begin{equation}
			\vec{w}(t) = (w_1(t), \ldots, w_M(t))^T.
		\end{equation} 
		The $N \times M$ matrix $D$ describes, how the unknown input signals 
		are distributed over the $N$ states $\vec{x}$. The rows of 
		$D$
		provide information about the inputs acting on the respective state. 
		Zero rows correspond to states which are not directly affected by 
		model errors. As discussed below, 
		the unknown input $\vec{w}$ incorporates genuine inputs from the exterior 
		as well as  all possible types of model errors, including 
		incorrect interactions between internal states~$\vec{x}$ and 
		incorrect parameter values. 
		The unknown input aka model error $\vec{w}$ needs to be estimated from data.
		This would not be a problem, if the internal system state 
		$\vec{x}(t)$ could directly be measured. However, typically only a smaller set of 
		$P$ scalar output signals
		\begin{equation}
			\vec{y}(t)=(y_1(t), \ldots, y_P(t))^T=\vec{c}\left(\vec{x}(t) \right)
		\end{equation}
		is directly accessible to measurements. The map from the state to the 
		output is here assumed to be given by the function 
		$\vec{c}\left(\vec{x}\right)$. We assume that both $\vec{f}$ and $\vec{c}$ satisfy
		a Lipschitz condition.
		In addition, the initial state $\vec{x}_0$ needs to be specified. 

		It is often useful to represent systems of the form in 
		Eq.~\eqref{eq:ss_model_affine} as an influence graph, see 
		Fig.~2(a) for an example.~The nodes of this directed 
		graph correspond to the states 
		$x_i,\,i=1,\ldots,N$ ,
		and the edges represent the interactions encoded by~$\vec{f}$. 
		More precisely, a directed edge 
		$x_i \to x_j$ 
		indicates, that 
		$\frac{\partial f_j}{\partial x_i} \ne 0$ 
		for some 
		$\vec{x}$ 
		in the state space 
		$\mathcal{X}$.~The 
		state nodes targeted by unknown inputs are determined by the 
		nonzero rows of the matrix 
		$D$. Please note, that these input nodes are sometimes also 
		called driver nodes, in particular in the context of control.
		In the special case were each unknown input component 
		$w_k,\,k\in\{1,2,\ldots,M\}$ 
		affects only one state node we can choose each element 
		$D_{ik}$ 
		to be either zero or one and we have 
		\begin{equation}\label{eq:Dmatrix}
			D_{ik} \in \{0,1\}  \tag{10a}
		\end{equation}
		and
		\begin{equation}
			\sum_{k=1}^M D_{ik} = \begin{cases}
			1 & \text{if input}\;  w_k \; \text{affects state}\, x_i\\
			0 & \text{otherwise}\, , \tag{10b}
			\end{cases} \setcounter{equation}{10}
		\end{equation}
		as indicated in Fig.~2(a) by red arrows. Similarly, if 
		the output function is given by the linear relationship 
		$\vec{c}\left(\vec{x} \right)= C\vec{x}$ 
		with the 
		$P \times N$-matrix 
		$C$
		and if, in addition, we have 
		$C_{ij}\in\{0,1\}$ 
		and 
		$\sum_{i}C_{ij} \in\{0,1\}$,
		we can indicate the subset of directly measured states by blue 	
		arrows. These correspond to the nonzero columns of $C$.

		\subsection{Model errors and unknown inputs}
		We now show that the  unknown input function $\vec{w}(t)$  
		in Eq.~\eqref{eq:ss_model_affine} can represent both the effect 
		of an exterior dynamic system to our open system $\mathcal{S}_{o}$ 
		and the effect of structural model errors in $\vec{f}$, i.e.~erroneous descriptions 
		of the interactions between the internal state variables.  
		
		Let us start with a closed system $\mathcal{S}$ modelled by
		\begin{equation}
			\dot{\vec{x}}(t) = \vec{f}\left(\vec{x}(t)\right),
			\label{eq:nominal_model}
		\end{equation}
		where only the internal dynamics of the system is described, but 
		unknown inputs are not taken into account. In reality, however, the system 
		$\mathcal{S}$
		might be embedded into a larger system (see~Fig.~\ref{fig:Fig1}) and 
		interact also with external state variables 
		$\vec{\xi}(t)=(\xi_1(t),\ldots,\xi_{N_{\text{e}}}(t))^T$, 
		which are not included in the model in~Eq.~\eqref{eq:nominal_model}. 
		Instead, the joint dynamics of the system $\mathcal{S}$ and its 
		exterior should rather be described by 
		\begin{subequations}
			\label{eq:big_system}
			\begin{eqnarray}
			\dot{\vec{x}}(t) &=& \vec{\phi}\left(\vec{x}(t), 
			\vec{\xi}(t)\right)	
			\label{eq:big_system_x}\\
			\dot{\vec{\xi}}(t) &=& \vec{\psi}\left(\vec{x}(t),
			\vec{\xi}(t)\right) \, .
			\end{eqnarray}
		\end{subequations}
		The function 
		$\vec{\phi}$
		combines the effect of interactions between internal states 
		$\vec{x}$ 
		and the effect of the external states 
		$\vec{\xi}$
		on the dynamics of the internal states. The dynamics of the external 
		system is determined by 
		$\vec{\psi}$, 
		which is usually not known and thus difficult to include in the 	
		model. This unknown dynamics leads to the structural model error $\vec{\eta}$, which 
		can formally be defined as the discrepancy between the model in 
		Eq.~\eqref{eq:nominal_model} and the system Eq.~\eqref{eq:big_system_x}
		\begin{equation}
			\vec{\eta}\left(\vec{x}(t), \vec{\xi}(t)\right) := \vec{\phi}
			\left(\vec{x}(t), \vec{\xi}(t)\right) -\vec{f}\left(\vec{x}(t)
			\right) \, .
			\label{eq:model_error}
		\end{equation}
		Thus, instead of explicitly extending the simpler 
		model~Eq.~\eqref{eq:nominal_model} to the potentially complicated joint 
		system~Eq.~\eqref{eq:big_system}, we could correct the former by 
		adding an unknown input
		\begin{equation}
			\vec{\eta}(t):=\vec{\eta}\left(\vec{x}(t), \vec{\xi}(t)\right)
		\end{equation}
		to obtain
		\begin{equation}
			\dot{\vec{x}}(t) = \vec{f}\left(\vec{x}(t)\right) + 
			\vec{\eta}(t) \, .
			\label{eq:hi_model}
		\end{equation}
		Replacing 
		$\vec{\eta}(t)=D \vec{w}(t)$ 
		in Eq.~\eqref{eq:hi_model} leads to Eq.~\eqref{eq:ss_model_affine}.  
		
		Let us illustrate the relationship between model errors and unknown inputs 
		by a simple concrete example. Consider the following model 
		\begin{subequations}
		\label{eq:prot_cascade}
		\begin{eqnarray}
			\dot{x}_1(t) &=& - 0.2\, x_1(t)\label{subeq:x1_pc}\\
			\dot{x}_2(t) &=& \frac{x_1(t)}{1+x_4(t)} -x_2(t) \label{subeq:x2_pc} \\
			\dot{x}_3(t) &=& x_2(t)-x_3(t)\label{subeq:x3_pc}\\
			\dot{x}_4(t) &=& x_3(t) -x_4(t)\label{subeq:x4_pc}
		\end{eqnarray}
		\end{subequations}	
		for a protein cascade~\cite{milo_network_2002} and assume that Eq.~\ref{subeq:x3_pc}
		in the model is misspecified. Instead, assume that the correct description is 
		\begin{equation}
		\dot{x}_3(t) = x_2(t)-\frac{x_3(t)}{a(t)+x_3(t)}\label{subeq:x3_pc_correct}\\
		\end{equation}
		with a time dependent function $a(t)$. Thus, the degradation of $x_3$ is an 
		enzyme catalysed reaction with a time dependent affinity $a(t)$, which can 
		not be described by the mass action term in Eq.~\ref{subeq:x3_pc}. In addition, 
		the affinity $a(t)$ is controlled by an exogenous regulatory process, which is
		not covered by the model in Eq.~\eqref{eq:prot_cascade}. The model error (compare Eq.~\eqref{eq:model_error}) 
		is given by
		\begin{eqnarray*}
		\vec{\eta} (t) &=& \left(0,\,0,\, x_3(t)-\frac{x_3(t)}{a(t)+x_3(t)},\, 0\right)^T \\
		               &=& \left( D w_1(t) \right)^T
		\end{eqnarray*}
		with $w_1(t)=x_3(t)-\frac{x_3(t)}{a(t)+x_3(t)}$ and $D=(0,0,1,0)^T$. If the correct 
		relationship in Eq.~\eqref{subeq:x3_pc_correct} is unknown, we can 
	    replace Eq.~\eqref{subeq:x3_pc} by 
		\begin{equation}
		\dot{x}_3(t) = x_2(t)-x_3(t)+w_1(t)\label{subeq:x3_pc_model_error}\\
		\end{equation}
		and treat $w_1(t)$ as an external input, which needs to be estimated from measurement
		data. An example for such an unknown input estimate is provided in Subsec.~\ref{ssec:practinv}.

\section{Criteria for invertibility}\label{sec:criteria_inv}

	If the unknown input function 
	$\vec{w}(t)$ 
	can uniquely be reconstructed from the output signal 
	$\vec{y}(t)$, 
	we call the system invertible.~An algebraic criterion to check for 
	invertibility of a system was first derived for linear systems
	\cite{
		sain_invertibility_1969},
	see Subsec.~\ref{subsec:rank_crit} for details.~An algorithm to invert 
	the system, which terminates in the case of a non-invertibility, was also 
	first devised for the linear case
	\cite{
		silverman_inversion_1969},
	but later extended 
	\cite{
		hirschorn_invertibility_1979}
	to nonlinear affine models of the form given in 
	Eq.~\eqref{eq:ss_model_affine}. Another type of results is based on 
	differential geometric or differential algebraic criteria for the invariant control 
	distributions
	\cite{
		nijmeijer_invertibility_1982, 
		nijmeijer_right-invertibility_1986, 
		fliess_note_1986}.
	All these criteria involve algebraic manipulations of the systems 
	equations, which makes them useful for smaller models, but limits their 
	utility for large networks. In addition, the invertibility tests require 
	a full specification of the system (Eq.~\eqref{eq:ss_model_affine}), 
	including the complete functional form and the parameters of the 
	interaction terms encoded by~$\vec{f}$. 

	Here, we state the exact mathematical definition for invertibility of 
	dynamic systems
	\cite{
		hirschorn_invertibility_1979}.
	Then, we provide a new recursive algorithm to check invertibility for
	linear systems, which might be easier to apply to systems of moderate
	size, in contrast to the mathematically equivalent algebraic rank condition
	\cite{
		sain_invertibility_1969}.
	However, for large networks, the structural invertibility algorithm has the huge 
	advantages of scaling to large systems and of only requiring the topology 
	of the influence graph. 
	Invertibility is only a necessary condition and the robustness of unknown
	input reconstruction to measurement noise might be influenced by other factors than just the network 
	structure. We briefly touch upon this important problem, discuss the role
	of regularisation and provide a simple, but illustrative example for an unknown 
	input estimate.
	
	\subsection{Invertibility}
		Mathematically, invertibility means that the map from the unknown 
		input signal 
		$\vec{w}$ 
		to the output 
		$\vec{y}$ 
		is injective, which can be expressed as
		\cite{
			hirschorn_invertibility_1979}:
		\begin{definition}\label{def:invertibility}
			The system~\eqref{eq:ss_model_affine} is \emph{invertible} at the 
			initial state 
			$\vec{x}_0$, 
			if two distinct input signals 
			$\vec{w}(t)$ 
			and 
			$\tilde{\vec{w}}(t)$ 
			always induce two distinct outputs 
			$\vec{y}(t)\ne \tilde{\vec{y}}(t)$. 
			If the system is invertible in an open neighbourhood of 
			$\vec{x}_0$, 
			it is called \emph{strongly invertible} at 
			$\vec{x}_0$.~The system is strongly invertible, if there 
			exists a dense submanifold 
			$\mathcal{M}$ 
			of 
			$\mathcal{X}$, 
			such that the system is strongly invertible for any 
			$\vec{x}_0 \in \mathcal{M}$.
		\end{definition}

		For linear systems 
		\begin{subequations}\label{eq:ss_model_lin}
			\begin{eqnarray}
			\dot{\vec{x}}(t) &=& A \vec{x}(t) + D \vec{w}(t) \\ 
			\vec{x}(0) &=& \vec{x}_0\\
			\vec{y}(t) &=& C \vec{x}(t)
			\end{eqnarray}
		\end{subequations}
		with a real 
		$N\times N$
		matrix 
		$A$
		and an 
		$P \times N$ 
		output matrix 
		$C$, 
		all the three definitions are equivalent
		\cite{
			hirschorn_invertibility_1979},
		since invertibility at some 
		$\vec{x}_0\in\mathbb{R}^N$ 
		implies invertibility at all points in their state space 
		$\mathcal{X}=\mathbb{R}^N$.
		Such linear systems are typically obtained as local approximations 
		of the nonlinear model in Eq.~\eqref{eq:ss_model_affine}, where 
		$A$ 
		and 
		$C$ 
		are given by the Jacobi matrices of 
		$\vec{f}$
	 	and 
	 	$\vec{c}$, 
	 	respectively, taken at a certain reference point. 

		For linear systems~(Eq.~\eqref{eq:ss_model_lin}), invertibility is a 
		global property and thus it is sufficient to consider 
		the initial condition 
		$\vec{x}_0=\vec{0}$. 
		Then, the input-output map 
		$\Phi$ 
		is given by the linear operator 
		\begin{equation}\label{eq:lin_io}
			\vec{y}(t) =  \left(\Phi \vec{w}\right)(t) = \int^t_0 C 
			\exp[{A(t-s)}] D \vec{w}(s) \, \text{d}s \, .
		\end{equation}
		The linear system is invertible, if this operator is one-to-one. 
		Below, we state two different versions 			of an algebraic criterion to 
		decide invertibility for linear systems.
		More details on the three criteria are given in the
		Appendix~\ref{app:Proof}. Here, we only  motivate 
		the structure of the algebraic criteria below: Taking successive 
		time derivatives
		 $\vec{y}^{(l)} (t) $
		of the output (Eq.~\eqref{eq:lin_io}), for $l\in\{1,2,\ldots\}$
		\begin{equation}\label{eq:output_deriv}
			\vec{y}^{(l)} (t) =  CA^l \vec{x}(t) + \sum_{k=0}^{l-1} C 
			A^{l-k-1} D \vec{w}^{(k)}(t) \, , 
		\end{equation}
		and evaluating at 
		$t=0$, 
		we obtain the sequence of linear equations
		\begin{eqnarray*}
			\dot{\vec{y}} (0) &=&  CD\vec{w}(0) \\
			\ddot{\vec{y}} (0) &=&   CAD\vec{w}(0) + CD\dot{\vec{w}}(0)\\
			\vdots& \\
			\vec{y}^{(l)} (0) &= &  CA^{l-1}D\vec{w}(0) + CA^{l-2}D
			\dot{\vec{w}}(0) + \\
			&&\quad \ldots + CD\vec{w}^{(l-1)}(0) \, . 
		\end{eqnarray*}
		Invertibility implies, that we can solve these linear equations 
		uniquely for 
		$\vec{w}^{(l)}(0)$
		for given 
		$\vec{y}^{(l)} (0)$, 
		or, equivalently, that 
		$\vec{y}^{(l)} (0)=\vec{0}$ 
		implies 
		$\vec{w}^{(l)}(0)=\vec{0}$. 
		Basically, the two equivalent algebraic criteria in Subsecs.~
		\ref{subsec:rank_crit} and~\ref{subsec:iter_crit} provide 
		conditions for unique solutions of this system. 
		
	\subsection{Invertibility versus Unknown Input Observability}
		Let us clarify the relationship between unknown input observability and invertibility. 		
		The notion of unknown input observability is not uniquely defined in the literature. 
		Some authors call a system unknown input observable, if the state is completely or partially 
		observable even in the presence of unknown inputs~(see e.g. \cite{martinelli_nonlinear_2019}).
		This weaker definition does not necessarily imply that the unknown input itself can be reconstructed. 
		Other authors define unknown input observability in the stricter sense that both the systems state as
		well as the unknown input can be inferred from the outputs~\cite{boukhobza_state_2007}. 
		
		Invertibility or structural invertibility ensures, that the unknown input can uniquely be 	
		reconstructed from the output.  It does not necessarily imply that the system state
		can simultaneously be observed~\cite{boukhobza_state_2007}.  Only if the initial state 
		$\vec{x}_0$ is known, then invertibility is not only necessary but also 
		sufficient for simultaneous state and input observability. 
		
		It is only recently, that a general algorithm for testing the weaker form of unknown input observability 
		of nonlinear systems exists~\cite{martinelli_nonlinear_2019}. This algorithm is based 
		on symbolic computation and thus restricted to very small systems with only a 
		few state variables. 	
		
	\subsection{Recursive Algebraic Criterion for Invertibility of Linear Systems}\label{subsec:iter_crit}
		For the first version of the algebraic criterion we define a sequence 
		of block matrices
		\begin{equation}
				R_l:=\left[ (CD)(CAD)\ldots (CA^lD)\right] \, , \quad
				l\in \mathbb{N}_0 \, ,
		\end{equation}
		which appear in the derivatives of the 
		output~(Eq.~\eqref{eq:output_deriv}). Each matrix 
		$R_l$ 
		in the sequence has 
		$P$ 
		rows and 
		$(l+1)M$ 
		columns. Recall, that 
		$P$ 
		was the number of measurement signals and 
		$M$ 
		the number of unknown inputs. For each matrix 
		$R_l$ 
		in the sequence, the null space is defined as
		\begin{equation}
			\ker R_l = \left\{\vec{v} \in \mathbb{R}^{(l+1)M} \left|\, R_l
			\vec{v} = \vec{0} \right. \right\} \, ,  \quad l\in  \mathbb{N}_0 \, .
		\end{equation}
		In addition, we recursively define the following sets 
		\begin{eqnarray*}
			K_0 &:=& \ker R_0\backslash \{0\}  \\
			K_l &:=& \ker R_l \cap \left(\mathbb{R}^M \times K_{l-1} \right)  \, .
		\end{eqnarray*}
		Here, ``$\times$'' indicates the cartesian product of two sets. Now we 
		can state the invertibility criterion: The linear system in 
		Eq.~\eqref{eq:ss_model_lin} is invertible, if and only if 
		$K_l$ 
		is the empty set 
		$\emptyset$ 
		for some 
		$l \leq N-1$. 

		To apply this criterion we have to calculate the null spaces of all 
		$R_l$ 
		and then iteratively to evaluate the sequence of sets 
		$K_l$, 
		starting from 
		$l=0$. 
		The iteration terminates if 
		$K_l=\emptyset$ 
		for some 
		$l \leq N-1$, 
		indicating invertibility. If no such 
		$l$
		can be found, the system is not invertible. 

			\begin{figure*}[htb]
			\centering
			\includegraphics[width=2\columnwidth]{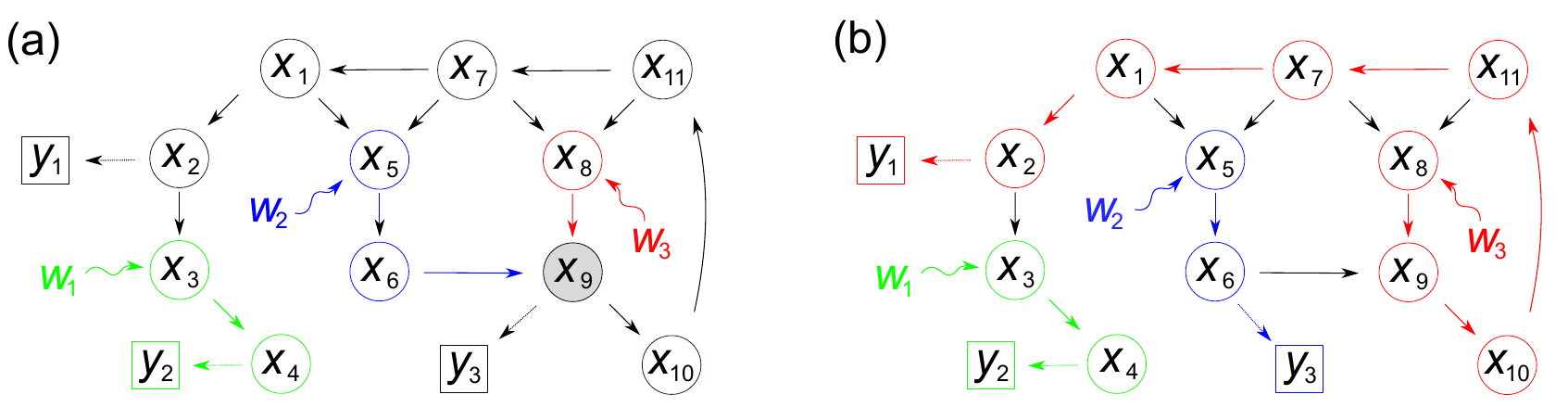}
			\caption{
				The graphical condition for invertibility. 
				(a) The system represented by the influence graph with input 
				node set 
				$S=\{x_3, x_5, x_8\}$ 
				is non-invertible, if the output measurements 
				$\vec{y} =(y_1,y_2,y_3)$ 
				(squares) are placed at the sensor node set 
				$Z = \{x_2, x_4,x_9\}$. 
				The only possible paths from the input nodes 
				$x_5$ 
				and 
				$x_8$ 
				to any of the outputs (here 
				$y_3$) 
				collide at 
				$x_9$. 
				(b) The same graph with the same input nodes 
				$S$ 
				as in (a), but with outputs 
				$\vec{y} =(y_1,y_2,y_3)$ 
				(squares) placed at 
				$Z=\{x_2, x_4, x_6\}$ 
				(i.e. 
				$x_6$ 
				replaces the sensor node 
				$x_9$ 
				in (a)) is invertible. There is a family 
				$\Pi = \{\textcolor{green}{w_1 \rightsquigarrow x_3 
				\rightarrow x_4 \rightarrow y_2},\, \textcolor{blue}{w_2 
				\rightsquigarrow x_5 \rightarrow x_6 \rightarrow y_3},\, 
				\textcolor{red}{w_3 \rightsquigarrow x_8 \rightarrow x_9 
				\rightarrow x_{10} \rightarrow x_{11} \rightarrow x_7 
				\rightarrow  x_1 \rightarrow x_2 \rightarrow y_1}\}$ 
				of  three non intersecting (node disjoint) paths joining each 
				input with one output.
				}
			\label{fig:Fig3}
		\end{figure*}

	\subsection{Rank Condition for Invertibility of Linear Systems}
	\label{subsec:rank_crit}
		The iterative criterion above is equivalent to the following 
		algebraic rank condition proofed by Sain and Massey
		\cite{
			sain_invertibility_1969}:
		Consider the sequence of matrices 
		\begin{equation}
			Q_l := \begin{bmatrix}
			CD & CAD & \dots & CA^lD \\ 
			0  & CD  & \dots & CA^{l-1}D \\
			\vdots&    &   \ddots   & \vdots  \\
			0     &  \dots &     &  CD
			\end{bmatrix}
		\end{equation}	
		for 
		$l\in \mathbb{N}_0$. 
		The linear system in Eq.~\eqref{eq:ss_model_lin} is invertible, if 
		and only if 
		\begin{equation}
			\text{rank}\,Q_{N-1} - \text{rank}\, Q_{N-2} = M\,.
		\end{equation}
		As before, $N$ is the number of states in the system and $M$ the number of inputs.
		Thus, the criterion requires to compute the rank of two matrices. The 
		size of these matrices increases quadratically with the number of 
		states 
		$N$ 
		in the system. Such rank computations can be very memory intensive 
		for large networks with many nodes.

		It is worthwhile remarking on an interesting property of invertibility.
		For an invertible system, the null space of the 
		operator 
		$\Phi$ 
		is zero dimensional, containing as its single element the zero input 
		$\vec{w}(t) \equiv \vec{0}$. 
		For a non-invertible system, the null space of 
		$\Phi$ 
		is always infinite dimensional~(see~Lemma~\ref{lemma:xiShift} in the 
		Appendix). This means that for non-invertible systems there are 
		infinitely many independent inputs which cannot be distinguished from each 
		other. This shows, that there is no such thing like 
		``nearly invertible''. Thus, any algorithm attempting to infer the inputs for 
		the outputs is bound to fail without further assumptions about the 
		inputs. Assumptions like smoothness and sparsity of the input signals 
		can be encoded into these inversion algorithms by using suitable 
		regularisation schemes
		\cite{
			engelhardt_learning_2016}
		or Bayesian priors
		\cite{
			engelhardt_bayesian_2017}.
		However, even these additional smoothness and sparsity assumptions 
		restricting the domain of the input-output map are not always 
		sufficient for invertibility 
		\cite{
			engelhardt_learning_2016}.

	\subsection{Graphical Criterion for Linear and Nonlinear Systems}\label{subsec:graphinv}
		\begin{figure*}[t]
			\centering
			\includegraphics[width = 2\columnwidth]{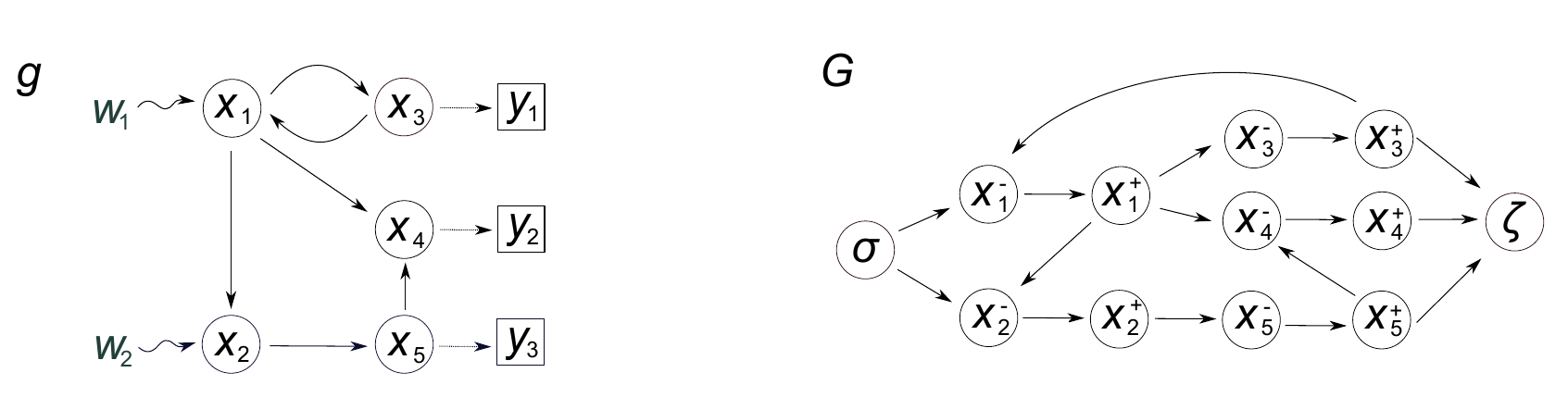}
			\caption{
				An example for the structural invertibility algorithm 
				counting the number of node-disjoint paths connecting the 
				input node set 
				$S=\{x_1,x_2\}$ 
				with the output node set 
				$Z=\{x_3,x_4, x_5\}$. The influence graph 
				$g$ 
				is transformed into the graph 
				$G$ 
				by replicating each of the state nodes 
				$x_1,\ldots, x_5$ 
				to 
				$x_1^+,\ldots, x_5^+$ 
				and 
				$x_1^-,\ldots, x_5^-$, 
				thereby separating the ingoing and outgoing edges. A family 
				of node-disjoint paths in 
				$g$ 
				corresponds to a family of edge-disjoint paths in 
				$G$. 
				Then, an additional source node 
				$\sigma$ 
				is connected to the inputs and an additional sink node 
				$\zeta$ 
				to the outputs and edge weights of 
				$1$ 
				are added to 
				$G$. 
				The maximum flow from 
				$\sigma$ 
				to 
				$\zeta$ 
				corresponds to the number of node-disjoint paths from 
				$S$ 
				to 
				$Z$ 
				in g.
				}
			\label{fig:Fig4}
		\end{figure*}

		The quite intricate algebraic conditions can be replaced by a simple 
		graphical criterion
		\cite{
			dion_generic_2003},
		see~Fig.~\ref{fig:Fig3}. Recall, that in the influence graph 
		$g$ 
		each state variable 
		$x_i$
		is represented as a node and the edges are determined by the 
		adjacency matrix 
		$A$: 
		For each 
		$A_{ji}\neq 0$ 
		draw an directed edge 
		$i\to j$. 
		Now assume, that the columns of the input matrix consist only of a 
		subset of 
		$M$
		canonical basis vectors 
		$\vec{e}_k\in\mathbb{R}^N$. 
		Thus 
		$D_{ik}\in\{0,1\}$ 
		with 
		$\sum_{k}D_{ik} \in\{0,1\}$.~Then, the nonzero rows of 
		$D$ 
		indicate the states receiving an input signal. Denote these 
		$M$ 
		\emph{input nodes} as
		\begin{equation}
			S=\{s_1,\ldots, s_M\}\subseteq\{x_1,\ldots,x_N\} \, .
		\end{equation} 
		Similarly, we assume that the output matrix 
		$C$ 
		has only 
		$P$ 
		rows which are a subset of the canonical basis of 
		$\mathbb{R}^P$ 
		and thus 
		$C_{ji}\in\{0,1\}$ 
		with 
		$\sum_{j}C_{ji} \in\{0,1\}$. 
		Then, the nonzero columns of 
		$C$ 
		indicate the 
		\emph{output} or \emph{sensor nodes} 
		\begin{equation}
			Z=\{z_1,\ldots,z_P\}\subseteq\{x_1,\ldots,x_N\} \, ,
		\end{equation}			
		i.e.~the state nodes for which direct measurements are available.   

		Now, the graphical criterion can be stated as follows: The linear 
		system in Eq.~\eqref{eq:ss_model_lin} is structurally invertible, if 
		and only if there is a family 
		$\Pi$ 
		of 
		$M$ 
		directed paths in the influence graph 
		$g$ 
		fulfilling the following conditions
		\begin{enumerate}
			\item[(i)] Each path in $\Pi$ starts in $S$ 
			and terminates in $Z$.
			\item[(ii)] All paths of $\Pi$ are pairwise node-disjoint.
		\end{enumerate}
		In the following, we will call the triplet 
		$(S,g,Z)$ 
		consisting of the influence graph 
		$g$, 
		the input node set 
		$S$ 
		and the output node set 
		$Z$ 
		invertible, if the structural invertibility criterion is fulfilled. 
 
		To put it differently: There must be a family 
		$\Pi$
		of directed paths in 
		$g$ 
		connecting each input node in 
		$S$ 
		with an output node in 
		$Z$ 
		and no two paths in the family intersect at any node of the influence 
		graph. If no such family of node-disjoint paths exists, the systems 
		is structurally non-invertible. For the example graphs in 
		Fig.~\ref{fig:Fig2}(a,b), this path condition is illustrated 
		in Fig.~\ref{fig:Fig3}(a,b) for the non-invertible (a) and invertible 
		case (b), respectively.  Obviously, a system with fewer output nodes 
		than input nodes is never invertible. 

		This intuitive graphical condition implies, that only the structure 
		of the influence graph and the position of the inputs and outputs, 
		i.e.~only the patterns of nonzero entries in the systems matrices 
		$A$, 
		$D$, 
		and 
		$C$, 
		are relevant for invertibility 
		\cite{
			dion_generic_2003}.~Indeed,
		structural invertibility and algebraic invertibility 	
		coincide up to pathological cases, where the graphical condition 
		could indicate structural invertibility, whereas none of the 
		algebraic conditions would be fulfilled due to an exact cancellation 
		of numerical terms.~These pathological conditions are irrelevant in 
		practice, since any arbitrarily small numerical perturbation of one 
		of the nonzero terms in the systems matrices would repair 
		invertibility. Or, in mathematical language: The set of systems 
		matrices 
		$A$, 
		$D$, 
		and 
		$C$, 
		for which the graphical and the algebraic conditions give 
		contradictory results is a set of measure zero. The situation is 
		completely analogous to the structural and algebraic controllability 
		or observability conditions
		\cite{
			lin_structural_1974, 
			dion_generic_2003, 
			liu_controllability_2011, 
			liu_observability_2013}.

		The structural invertibility condition was extended to
		nonlinear systems of the form in Eq.~\eqref{eq:ss_model_affine}, 
		see
		\cite{
			wey_rank_1998}
		for details. This means, that we can replace the matrix 
		$A$ 
		by the Jacobi matrix of 
		$\vec{f}$ 
		and the output matrix 
		$C$ 
		by the Jacobi matrix of 
		$\vec{c}$ 
		at some point of the state space 
		$\mathcal{X}$ 
		to obtain the systems graph and the output node set 
		$Z$. Thus, the structural properties of 
		the linearisation of Eq.~\eqref{eq:ss_model_affine} are also sufficient to 
		detect the invertibility of a nonlinear system. There is one subtlety
		for nonlinear systems: The structural invertibility condition does 
		not imply regularity of the system, which is relevant
		for feedback systems, see~\cite{wey_rank_1998} for further details.

	\subsection{Structural Invertibility Algorithm}\label{ssec_sialg}
		The graphical criterion for (structural) invertibility requires to 
		count the number of node-disjoint paths connecting the input nodes 
		$S$ 
		with the output nodes 
		$Z$. 
		Counting all these paths in a combinatorial manner is not feasible 
		for larger systems. Thanks to the Max-Flow-Min-Cut-Theorem 
		\cite{menger_zur_1927, 
		dion_generic_2003, 
		korte_combinatorial_2018}, the graphical node disjoint path counting problem 
		can be reformulated as flow problem, which can efficiently be solved. 
		
		As an initial step of the algorithm, the influence graph 
		$g$ 
		is transformed to a corresponding flow graph 
		$G$ 
		by copying each node 
		$i$ 
		to separate the ingoing and outgoing edges (see Fig.~\ref{fig:Fig4}). 
		Now, a familiy 
		$\Pi$ 
		of node-disjoint paths in the original graph 
		$g$ 
		corresponds to a family of edge-disjoint paths in 
		$G$. 
		In a second step, an additional source node 
		$\sigma$ 
		is connected to each of the input nodes and an  additional sink node 
		$\zeta$ 
		is connected to each of the output nodes. If each edge in the 
		resulting graph 
		$G$ 
		is assigned a weight of 
		$1$, 
		the maximum flow from source 
		$\sigma$ 
		to sink 
		$\zeta$ 
		in 
		$G$ 
		equals the number of edge-disjoint paths from source to sink, and 
		thus the number of node-disjoint paths from 
		$S$ 
		to 
		$Z$ 
		in the original graph 
		$g$. In our implementation we use the Goldberg-Tarjan algorithm~\cite{goldberg_tarjan_1988}
		to efficiently compute the maximum flow. Several alternative algorithms exist in 
		the combinatorial optimisation literature, see e.g.
		\cite{
			korte_combinatorial_2018}.

		To analyse the computational complexity of the structural invertibility algorithm,
		let $N$ and $E$ be the number of nodes and edges in the original graph $g$. As before, 
		we denote the number of input nodes by $M$ and the number of sensor nodes by $P$. In directed graphs 
		the number of edges is limited by $E \leq N^2$. For large networks
		we can assume $M\approx P  << N$. To create the flow graph~$G$, $n=2N+2$ nodes and $e=N+E+M+P$ edges are created. 
		On $G$ we use the Goldberg-Tarjan algorithm with running time scaling like  
		$\mathcal{O}(n^2 \sqrt{e})$ to compute the maximum flow. All together we find that
		the structurally invertibility algorithm has a running time of~$\mathcal{O}(N^3)$. Please
		note that there are even more efficient optimised versions of the Goldberg-Tarjan algorithm 
		with better running time~%
		\cite{
			korte_combinatorial_2018},
		but for all our purposes the standard version was sufficient. 
		
		Our implementation is based on the \texttt{python networkx} package. On a single node
		of an \texttt{Intel® Xeon® Processor E5-2690 v2} we need an average time of $0.15 \pm 0.03$ 
		seconds for a network with $N=10^3$ nodes and $2.4 \pm 0.4$ seconds for $N=10^4$ nodes to 
		decide structural invertibility. 		
		
	\subsection{Practical invertibility and robustness}\label{ssec:practinv}
		Invertibility (or structural invertibility) is a necessary condition for the unique 
		reconstruction of unknown inputs $\vec{w}(t)$ from outputs $\vec{y}(t)$. 
		If the input-output map $\Phi$ defined by the 
		general system in Eq.~\eqref{eq:ss_model_affine} is invertible, the operator equation 
		\begin{equation}
		\vec{y}=\Phi \vec{w} \label{eq:operator_eqn}
		\end{equation}
		for given $\vec{y}=\vec{y}(t)$ has a unique solution $\vec{w}(t)$ and the 
		inverse operator $\Phi^{-1}$ exists.
		
		In reality, we have to reconstruct the unknown input $\vec{w}$ from measured output data 
		$\vec{y}^{\text{data}}$, which will always be subject to measurement errors and noise. Therefore,
		the data based estimate $\hat{\vec{w}}=\Phi^{-1}\vec{y}^{\text{data}}$ will differ from the 
		true input $\vec{w}=\Phi^{-1}\vec{y}$. For a discontinuous inverse operator~$\Phi^{-1}$,
		the difference between $\hat{\vec{w}}$ and $\vec{w}$ can be drastic. Unfortunately,
		the inverse operator of the general nonlinear system Eq.~\eqref{eq:ss_model_affine} is not 
		continuous. Thus, estimating the unknown input from real data remains an ill posed 
		problem~\cite{engl_regularization_2000, nakamura_potthast_inverse_2015}, even if the system is invertible and 
		$\Phi^{-1}$ exists.
		
		The underlying reason for the discontinuity of the inverse input-output operator $\Phi^{-1}$ 
		of the linear system (Eq.~\eqref{eq:ss_model_lin}) is a well known 
		theorem from the theory of inverse problems~(see e.g. \cite{engl_regularization_2000, nakamura_potthast_inverse_2015}):
		Linear compact operators with an infinite dimensional range cannot have a continuous inverse.
		The linear input-output operator Eq.~\eqref{eq:lin_io} corresponding to the linear system in 
		Eq.~\eqref{eq:ss_model_lin} is an integral operator and thus it is compact and has infinite range, 
		as shown in the Appendix (Theorem~\ref{lemma:xiShift}). For nonlinear operators, a similar 
		theorem states that completely-continous operators with infinite range cannot have a continuous
		inverse. This indicates, that the inverse of the nonlinear system Eq.~\eqref{eq:ss_model_affine} 
		cannot be continuous.
		
		The degree of discontinuity of the inverse to the linear compact operator in Eq.~\eqref{eq:lin_io} 
		can be quantified by means of the  singular value decomposition (SVD). Since $\Phi$ has an 
		infinite dimensional range, its SVD is an infinite series. The infinite series 
		$(\sigma_k), \,k=1,2,\ldots$ of singular values is usually ordered by decreasing 
		magnitude ($\sigma_{k+1}\ge\sigma_k)$. The smaller singular values determine the response of~$\Phi^{-1}$ 
		to high frequency components of $\vec{y}^{\text{data}}$ and are thus responsible for the 
		discontinuity of the inverse. Thus, it is in principle possible to quantify the degree 
		of discontinuity by the speed of decay of the sequence~$(\sigma_k)$. The problem in 
		Eq.~\eqref{eq:operator_eqn} is considered to be mildly ill-posed, if the sequence of singular 
		value $(\sigma_k)$ decays at most with polynomial speed, whereas it is called 
		severely ill-posed, if $(\sigma_k)$ decays faster than any polynomial~\cite{engl_regularization_2000}. 	
		However, this approach is not  straightforward to implement, because it requires computing the 
		spectrum of the gramian operator $\Phi^*\,\Phi$, where $\Phi^*$ is the adjoint of 
		Eq.~\eqref{eq:lin_io}. In addition, it is not
		yet clear whether the SVD can also be useful for the inverse nonlinear input-output operator
		corresponding to nonlinear system Eq.~\eqref{eq:ss_model_affine}, possibly after a suitable 
		linearisation. This is certainly beyond the scope of this text and we leave this as an interesting 
		direction for further research. 
		
		Please note, that the situation for invertibility is different from that of controllability or 
		observability~
		\cite{ 	krener_meas_2009,
			cornelius_controlling_2011, 
			sun_controllability_2013, 
			cornelius_realistic_2013, 
			yan_spectrum_2015,
			summers_submodularity_2016, 
			aguirre_observ_2018,
			haber_state_2018},
		where the corresponding gramian matrices correspond to operators with 
		finite dimensional range. Consequently, there is a smallest singular value which can 
		be used as the condition number characterising the degree of controllability or observability, 
		respectiveley. 
		
		\subsubsection*{Regularisation of Unknown Inputs}
		There are several algorithms for estimating the unknown input $\vec{w}$
		from measurement data, ranging from feedback controllers via modifications of the nonlinear Kalman filter to 
		moving horizon estimation~
		\cite{  kuhl_real-time_2011,
			schelker_comprehensive_2012,
			Fonod_boblin_unknown_2014, 
			engelhardt_learning_2016, 
			engelhardt_bayesian_2017, 
			chakrabarty_state_2017, 
			tsiantis_optimality_2018}.
		We can not discuss all these approaches here, but it is instructive to 	
	    briefly discuss a simple version of the optimisation based approach, where an error functional $J[\vec{w}]$ is minimised
		with respect to $\vec{w}(t)$. This leads to the following 
		optimal control problem
		\begin{equation}\label{eq:estopt}
		\begin{aligned}
			&\text{minimize} \, J[\vec{w}] = d[\vec{y}^\text{data},\vec{y}] + R[\vec{w}] \\
			&\text{subject to } \\
			&\dot{\vec{x}}(t) = \vec{f}(\vec{x}(t)) + \vec{w}(t) \\
			&\vec{y}(t) = \vec{c}(\vec{x})\,.
		\end{aligned}	
        \end{equation}		
		Here, $d$ quantifies the fit of the corrected 
	    model output $\vec{y}$ to the data~$\vec{y}^
	    \text{data}$. A typical choice is the squared error
	    \begin{eqnarray*} 
		d[\vec{y}^\text{data},\vec{y}] &=& \frac{1}{T}\sum_{i=1}^P \int_0^T 
		\left(y_i^\text{data}(t)-
		y_i(t) \right)^2 \text{d}t \\
		&\approx& \frac{1}{n}\sum_{i=1}^P \sum_{k=1}^n  
		\left(y_i^\text{data}(t_k)-
		y_i(t_k) \right)^2\,.
		\end{eqnarray*}
		Usually, one has discrete time measurements
    	$\vec{y}^\text{data}(t_k)$ at time points $t_k\in [0,T],\,k=0,1,\ldots,n$ and the discrete time 
    	squared error in the second row is used instead of the integral. If the components of the 
    	output function have very different magnitudes, it is often also useful to use a weighted
    	squared error. In addition, for zero mean gaussian measurement noise $\vec{y}^{\text{data}}-\vec{y}$,
    	the squared error corresponds to the log-likelihood function~\cite{nakamura_potthast_inverse_2015}.

    	The regularisation term $R[\vec{w}]$ in Eq.~\eqref{eq:estopt} can be chosen to
    	penalise overly complex input functions $\vec{w}(t)$.
    	The regularisation 
    	parameter $\alpha \ge 0$ provides a
    	way of balancing the data fit ($d$) with the complexity of the estimated function $\vec{w}(t)$.
    	There are several ways to select the regularisation parameter. One useful idea is known as the
    	discrepancy principle~\cite{nakamura_potthast_inverse_2015}, where the regulation parameter
    	is chosen such that the data error $d$ is approximatl	y equal to the level of measurement noise. 
  
  		Even for invertible systems the regularisation is necessary, because the unregularised least 
    	squares fit ($\alpha=0$) 
    	will be very sensitive to measurement errors, again a symptom of the discontinuity of the 
    	inverse input-output operator~$\Phi^{-1}$ corresponding to Eq.~\eqref{eq:ss_model_affine}.
    	Typical choices for the regularisation function are
    	\begin{equation*}
    	R[\vec{w}] = \frac{1}{T} \sum_{i=1}^p  \int_0^T w_i^2(t) \text{d}t\,,
    	\end{equation*}
    	which penalises the squared 2-norm of the unknown input $\vec{w}$ or
    	\begin{equation}
    	R[\vec{w}] = \frac{1}{T} \int_0^T \dot{w}^2_i(t) \text{d}t\,,\label{eq:reg_deriv}
    	\end{equation} 
    	which penalises the norm of the first derivative $\dot{\vec{w}}$. These two examples
    	are known as Tikhonov regularisation in function spaces~\cite{nakamura_potthast_inverse_2015}. 
    	For linear operators
    	like the Eq.~\eqref{eq:lin_io}, the effect of the regularisation is to suppress the effect
    	of small singular values.  There are many more possible choices for regularisation terms,
    	e.g. sparsity promoting regularisation~\cite{engelhardt_learning_2016}. 
    	Often, regularisation terms are also chosen to render the optimisation problem convex, 
    	which avoids problems with local minima~\cite{abarbanel_predicting_2013}. 
    	
    	\subsubsection*{Example for an unknown input estimate}
    	In our computational experiments 
    	we found the regularisation of the derivative as in Eq.~$\eqref{eq:reg_deriv}$ to yield 
    	good results for the case of additive measurement noise, where 
    	$\vec{y}^{\text{data}}-\vec{y}$ is a zero mean stochastic disturbance of the 
    	output measurement. This is illustrated in Fig.~\ref{fig:Fig5} for the model of a 
	protein cascade~\cite{milo_network_2002} 
	\begin{equation}
		\begin{aligned}
			\dot{x}_1(t) &= -0.2\, x_1(t) +w_1(t)\\
			\dot{x}_2(t) &= \frac{x_1(t)}{1+x_4(t)} -x_2(t) \\
			\dot{x}_3(t) &= x_2(t)-x_3(t)+w_2(t)\\
			\dot{x}_4(t) &= x_3(t) -x_4(t)\,,
		\end{aligned} \label{eq:ExampleSystem}
	\end{equation}
	with known initial value
	\begin{equation}
		\vec{x}_0 = (1,0.5,0.5,0.1)^T \, .
	\end{equation}
	We generated synthetic data by solving the system of ODEs for a given 
	input function $\vec{w}(t)$ acting on the input nodes $S=\{x_1,x_3\}$ 
	and adding gaussian noise to the output, see 
	Fig.~\ref{fig:Fig5}(a,b). We then tried
	to recover the input from the noisy output data by solving the regularised optimal 
	control problem in Eqs.~\eqref{eq:estopt} and \eqref{eq:reg_deriv}. 
	If the measurement nodes are 
	$Z=\{x_3,x_4\}$, the measured output data can accurately be fitted~(Fig.~\ref{fig:Fig5}(b)),
	but the recovered inputs~\ref{fig:Fig5}(c) do not correspond to the true inputs. The 
	reason is that the system with the output nodes $Z=\{x_3,x_4\}$ is not structurally 
	invertible. For the sensor node set $Z=\{x_2,x_4\}$, the system is structurally invertible
	and the solution of the optimal control problem (Eqs.~\eqref{eq:estopt} and \eqref{eq:reg_deriv}) 
	provides a reconstruction of the inputs from the noisy measurements, see Fig.~\ref{fig:Fig5}(d-f).
	In particular, the magnitude of the estimate for $w_1$ is very small and 
	differs from the true unknown input $w_1=0$ only 
	due to transient effects of the numerics and measurement noise. 
	\begin{figure*}
			\centering
			\includegraphics[width = 2\columnwidth]{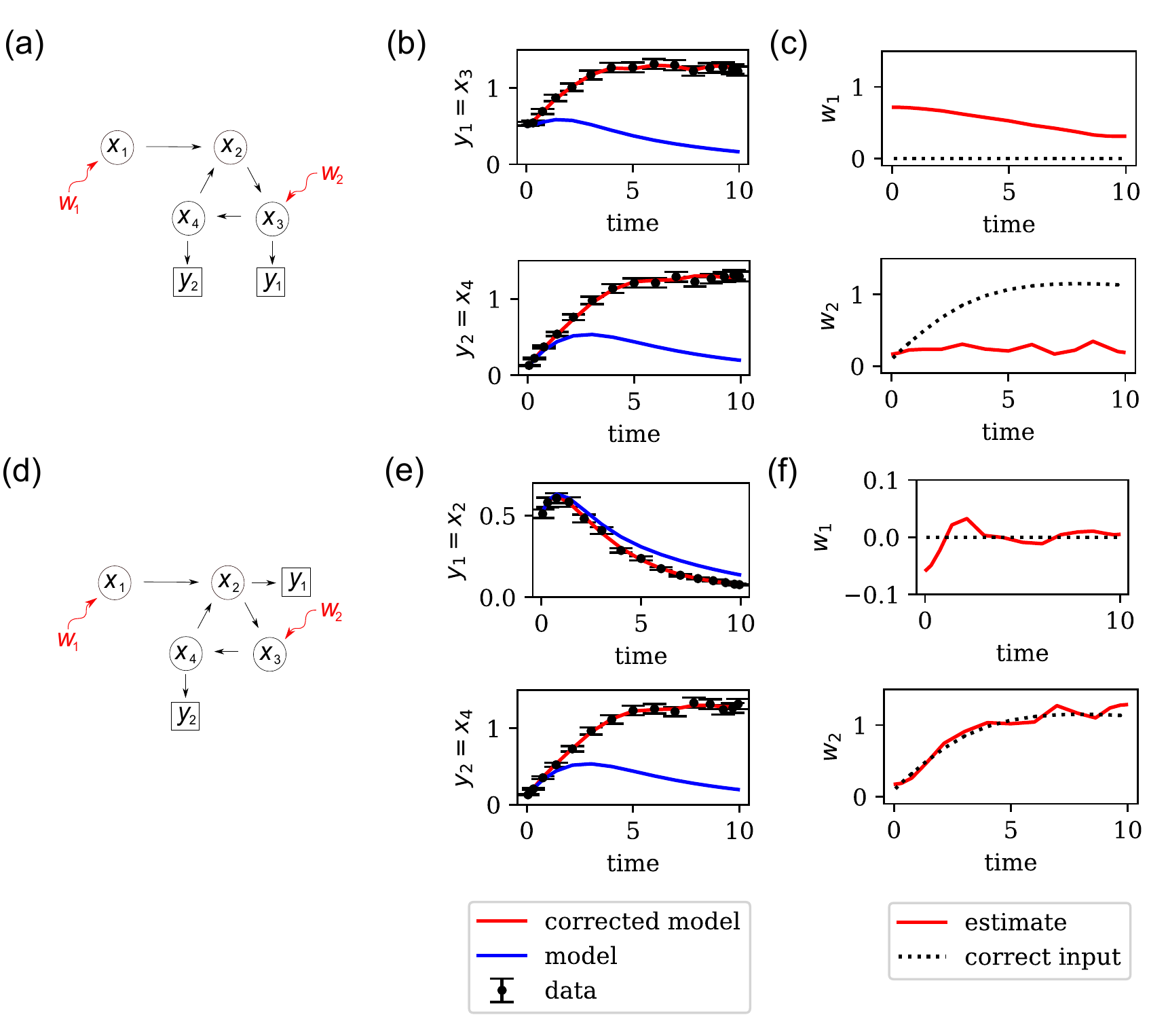}
			\caption{Estimation of unknown inputs to a simple protein cascade. 
				(a) The influence graph of the system in Eq.~\eqref{eq:ExampleSystem} with 
				input nodes $S=\{x_1, x_3\}$ and output nodes $y_1=x_3$ and $y_2=x_4$.
				(b) The model output for the system without unknown inputs (blue) can 
				not reproduce the observed data (dots with error bars for measurement noise).
				Solving the regularised optimisation problem~(see Eqs.~\eqref{eq:estopt} and
				\eqref{eq:reg_deriv}) provides accurate fits to the data (red line).  
				(c) However, the estimates of the unknown inputs (red lines) 
				are incorrect (dashed lines indicate the true inputs), 
				because the system is not structurally invertible. 
				(d) The system is structurally invertible, if $y_1=x_2$ and $y_2=x_4$ are 
				measured. (e,f) The fit to the measurements in (e) is now sufficient to 
				estimate the unknown inputs in (f). Please note also 
				the different scale of the plots for in (c) and (f). All values are understood in arbitrary units. 
				}
			\label{fig:Fig5}
		\end{figure*}
		
		This simple example illustrates again the importance of structural invertibility
		as a prerequisite for systems inversion. As discussed above, the accuracy of the 
		estimates can vary with the degree of continuity of the inverse input-output
		operators, which itself depends on the functional form and on the specific parameters of the 
		ODE. However, structural invertibility is a necessary requirement
		to estimate unknown inputs or model errors. Therefore, we will focus on structural 
		aspects of invertibility in the remainder of this text.

\section{Structural Invertibility of Complex Networks}\label{sec:net}
    It is natural to ask whether certain network properties affect structural invertibility.
	It has been shown previously, that important 
	systems properties including controllability
	\cite{
		liu_controllability_2011},
	observability 
	\cite{
		liu_observability_2013} 
	or target controllability 
	\cite{
		gao_target_2014} 
	are related to network structure, see
	\cite{
		motter_networkcontrology_2015, 
		liu_control_2016} 
	for reviews. 

	In this section we will explore the invertibility of large simulated and 
	real networks using  the very efficient structural invertibility 
	algorithm from Subsec.~\ref{ssec_sialg}. To mimick scenario SC~\Romannum{1}, where we have no influence on the choice of the input and 
	output nodes, we will first use a uniform random sampling scheme for the 
	selection of both node sets. Later, we will investigate the effect of 
	hubs on invertibility and study the preferential selection of hubs as 
	input or output nodes. 

	\subsection{Invertibility of random and scale free networks under uniform sampling~(scenario~SC~\Romannum{1})}\label{subsec:unif}
		\begin{figure*}
			\centering
			\includegraphics[width = 2\columnwidth]{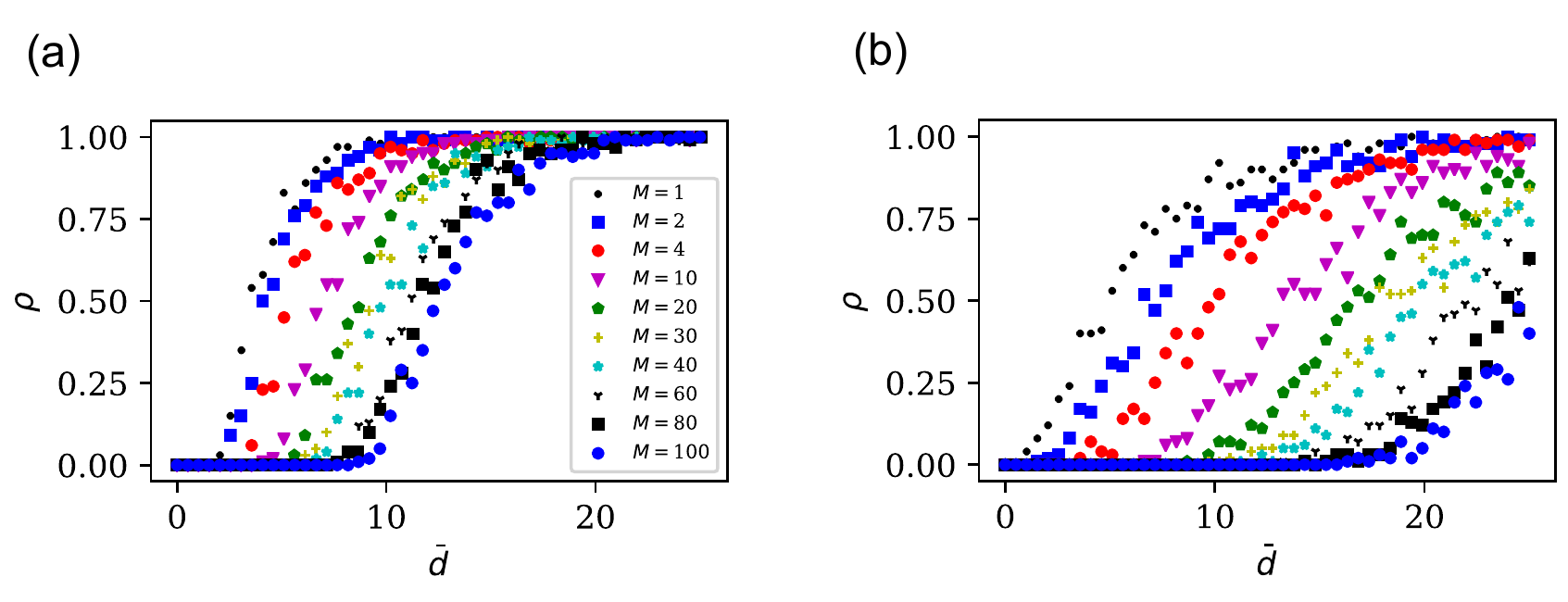}
			\caption{
				The structural invertibility of (a) Erd\H{o}s-R\'enyi random 
				graphs and (b) scale free networks
				(power law exponent 
				$\gamma = 2.4$) 
				depends on the average node degree 
				$\bar{d}$. 
				For each data point, an ensemble of 
				$100$
				networks with 
				$N=10^3$ 
				nodes was simulated and disjoint sets of 
				$M=P$ 
				input and output nodes were chosen by uniform random 
				sampling. Each network was tested for structural 
				invertibility and the fraction 
				$\rho$ 
				of invertible networks in the ensemble was recorded. For 
				large networks with many nodes 
				$N \to \infty$ 
				and fixed average degree 
				$\bar{d}$, 
				the function 
				$\rho(\bar{d},N, M)$ 
				asymptotically approaches 
				$\rho(\bar{d}, M)$. 
				We found empirically that 
				$N=10^3$ 
				is a good approximation for this asymptotic regime.
				}
			\label{fig:Fig6}
		\end{figure*}
		Intuitively, a densely connected network allows to find many 
		node-disjoint paths connecting the input node set 
		$S$ 
		to the output node set 
		$Z$. 
		Thus, for a set of randomly selected input nodes 
		$S$ 
		and a disjoint set of randomly selected output nodes 
		$Z$, 
		the chance for invertibility should increase  with the average 
		degree 
		$\bar{d}$ 
		of the network 
		$g$.  

		To test this hypothesis, we simulated graphs with 
		$N=10^3$ 
		nodes using either Erd\H{o}s-R\'enyi random graphs
		\cite{
			erdhos_evolution_1960}
		or scale-free networks 
		\cite{
			barabasi_emergence_1999, 
			goh_universal_2001} 
		with varying average degree 
		$\bar{d}$. 
		Throughout these simulations, we used 
		$M=P$, 
		i.e. the same number of input and output nodes. For a given graph 
		$g$, 
		we first sampled a set 
		$S$ 
		of 
		$M$ 
		input nodes uniformly at random and then randomly sampled the set 
		$Z$ 
		of 
		$P$ 
		output nodes from the remaining nodes, such that the input and output 
		node sets are disjoint: 
		$S \cap Z = \emptyset$. 
		We will refer to this sampling scheme for inputs and outputs as 
		uniform random sampling, which simulates scenario~SC~\Romannum{1}. 
		Then, we used the structural invertibility algorithm to check whether 
		the resulting network represents the influence graph of an invertible 
		or a non-invertible system~(Eq.~\eqref{eq:lin_io}).~To estimate the 
		probability 
		$\rho=\rho(\bar{d},N, M)$, 
		that a graph with $N$ nodes, 
		$M=P$ 
		inputs and outputs and average degree 
		$\bar{d}$ is invertible under this random scheme, we sampled 
		$100$
		triples of 
		$(S,g,Z)$ 
		of different graphs and input/output node sets and counted the 
		relative frequency of structurally invertible systems represented by 
		these graphs.  

		As can be seen in Fig.~\ref{fig:Fig6}(a), the probability of 
		invertibility for Erd\H{o}s-R\'enyi random graphs increases indeed 
		monotonously with the average degree 
		$\bar{d}$. 
		For small 
		$\bar{d}$, 
		almost no graph is invertible, whereas for large 
		$\bar{d}$ almost all graphs are invertible.~These two regimes are 
		separated by a transition zone, where some networks are invertible 
		and others are not. In this transition zone, invertibility depends on 
		the specific characteristics of the random graph and the average 
		degree is not sufficient to decide  about invertibility. For more 
		inputs and outputs (increasing 
		$M=P$), 
		the transition zone moves towards higher 
		$\bar{d}$. 
		This is plausible, because a family of 
		$M=P$ 
		node disjoint paths 
		$\Pi$ 
		connecting the input and the output sets 
		$S$ 
		and 
		$Z$ 
		cannot be found in sparse networks with a small overall number of 
		paths. We found empirically, that for 
		$M=P$ 
		the function 
		$\rho=\rho(\bar{d},N, M)$ 
		attains an asymptotic limit 
		$\rho(\bar{d}, M)$ 
		for large networks with a given average 
		degree ($N \to \infty$
		and 
		$\bar{d}$ 
		fixed).  

		Scale-free networks offer another network topology induced by a power 
		law degree distribution 
		$P(k)\propto k^{-\gamma}$
		that has been observed to be the underlying structure of 
		many real networks 
		\citep{
			barabasi_emergence_1999, 
			goh_universal_2001}.~Scale-free 
		networks have a tendency for a few highly connected 
		\textit{hubs} and many weakly connected \textit{satellites}.~The 
		effect of this heterogeneity is not immediately obvious: On the one 
		hand, the hubs act as bottlenecks, that shrink the chance of finding 
		node-disjoint paths. On the other hand, the diameter of scale-free 
		networks is much smaller than the diameter of Erd\H{o}s-R\'enyi 
		random graphs 
		\cite{
			cohen_scale-free_2003};
		hence paths are 
		shorter and might possibly find their way to the output set 
		$Z$, 
		before they intersect. 
		
		The \texttt{python networkx} package was used to do the simulations. For the 
		Erd\H{o}s-R\'enyi graphs $g=g(N,p)$ we used \texttt{fast\_gnp\_random\_graph}. 
		We implemented  the static model from 
		\cite{
			goh_universal_2001}
		to generate scale-free graphs $g=g(\bar{d},N,\gamma)$.
		Here, $N$ is the node number, $p$ the probability, that an edge in the 
		Erd\H{o}s-R\'enyi graph exists, $\bar{d}$ is the average degree
		and $\gamma$ the exponent 
		in the power law degree distribution. As before we drew 
		$100$ 
		graphs, distributed 
		$M=P$ 
		input and output nodes uniformly over each graph, and took 
		$\rho$ 
		as the fraction of structural invertible graphs. For a given number 
		of inputs and outputs, the transition zone  for scale free networks 
		(see~Fig.~\ref{fig:Fig6}(b)) is broader in comparison to 
		Erd\H{o}s-R\'enyi systems. In scale free networks, increasing the 
		number of inputs and outputs has a more drastic effect on diminishing 
		the chance for invertibility, as can be seen from the larger gaps 
		between the different curves in~Fig.~\ref{fig:Fig6}(b) compared 
		to Fig.~\ref{fig:Fig6}(a). For the same average degree 
		$\bar{d}$, 
		one is less likely to sample an invertible combination of inputs and 
		outputs in a scale free network than in a homogenous 
		Erd\H{o}s-R\'enyi random graph. Thus, under the uniform random 
		placement scheme (scenario~SC~\Romannum{1}) of inputs and outputs, 
		hubs are typically detrimental for invertibility. 

	\subsection{The role of the degree distribution}
		\begin{figure*}
			\centering
			\includegraphics[width = 2\columnwidth]{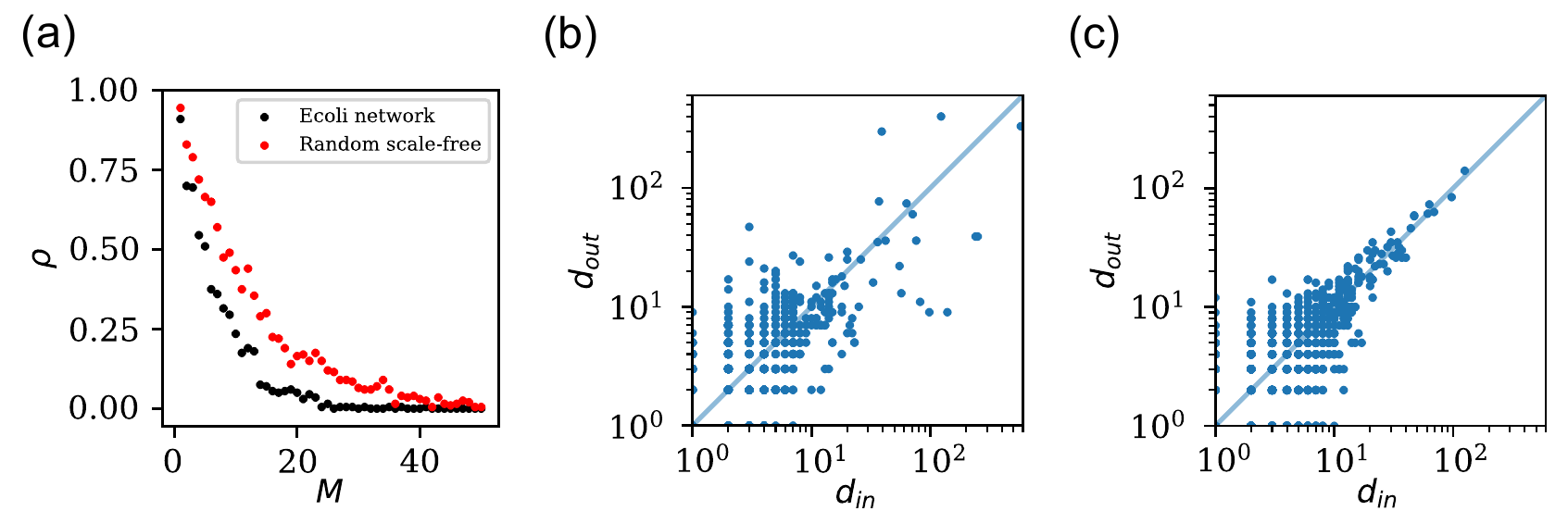}
			\caption{
				Effect of the degree distribution on invertibility under the 
				uniform random scheme. (a) The fraction of structural 
				invertible input-output configurations for the 
				\textit{E.coli} metabolic network (black dots) is compared to 
				an ensemble of scale free random networks with the same power 
				law exponent 
				$\gamma = 2.61$ 
				and the same average degree 
				$\bar{d}= 11.17$. 
				(b,c) The out-degree versus the in-degree for the 
				\textit{E.coli} metabolic network (b) and a typical scale 
				free network (c).
				}
			\label{fig:Fig7}
		\end{figure*}
		To explore the effect of the degree distribution on invertibility, we 
		compared the scale free \textit{E.coli} metabolic network
		\cite{
			schellenberger_bigg:_2010}
		to an ensemble of simulated scale free networks. The \textit{E.coli} 
		metabolic network has an estimated power law exponent of 
		$\gamma=2.61$ 
		and an average degree of 
		$\bar{d}=11.17$. 
		We used the static model and the same parameters to simulate the 
		ensemble of 100 scale free networks. We selected 100 input and output 
		node sets for the \textit{E.coli} network using uniform random 
		sampling (scenario~SC~\Romannum{1}) and computed the fraction 
		$\rho$ 
		of invertible systems as a function of the number 
		$M=P$ 
		of in- and outputs.~The uniform random sampling scheme was also 
		applied to each of the 100 simulated scale free graphs. Intriguingly, 
		we found that  the probability for invertibility 
		$\rho$ 
		is higher in the simulated networks than in the \textit{E.coli} 
		metabolic network, see Fig.~\ref{fig:Fig7}(a). In addition, we 
		performed a degree-preserving randomization (rand-Degree) 
		\cite{
			maslov_specificity_2002}
		to all networks (\textit{E.coli} and simulated) and found that this 
		doesn't change 
		$\rho$, 
		up to small sampling deviations (see next 
		Subsec.~\ref{subsec:real1}). In this degree-preserving 
		randomization, the in-degree 
		$d_\text{in}$ 
		(number of incoming edges) and the out-degree 
		$d_\text{out}$ 
		(number of outgoing edges) of each node is preserved, but the nodes 
		which link to each other are randomly selected. 

		In Fig.~\ref{fig:Fig7}(b,c) we have plotted 
		$d_\text{out}$ 
		versus 
		$d_\text{in}$ 
		for the \textit{E.coli} metabolic network (b) and a typical simulated 
		scale free network (c). It can be seen, that the \textit{E.coli} 
		network has many more high degree nodes with a large difference 
		$d_\text{out}-d_\text{in}$ 
		between out- and in-degree. This asymmetry is by construction much 
		smaller in the simulated networks. These results indicate, that the 
		joint distribution of in- and out-degrees 
		$P(d_{in}, d_{out})$ 
		largely determines the probability of finding an invertible system 
		under the uniform input-output sampling scheme 
		(scenario~SC~\Romannum{1}). 
		\begin{figure}
			\centering
			\includegraphics[width = \columnwidth]{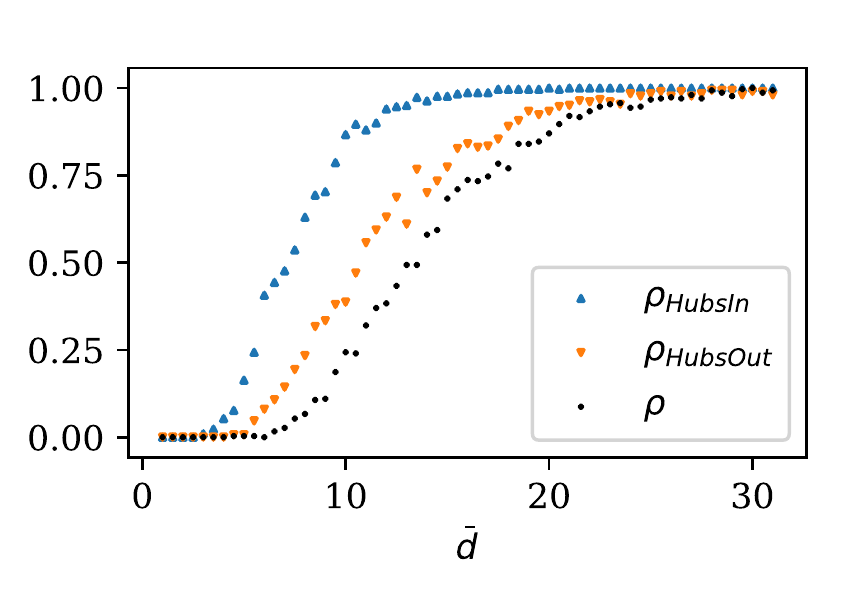}
			\caption{
				Preferential sampling of hubs as input nodes (blue) or output 
				nodes (orange) for an ensemble of 
				$300$
				scale free networks 
				($N=500$, 
				$M=P=10$, 
				$\gamma=2.4$). 
				The blue triangles show the fraction of invertible networks, 
				when the 
				$M=10$ 
				nodes with highest degree are chosen as input nodes and the 
				$P=10$ 
				output nodes are sampled uniformly from the remaining 
				$N-M=490$ 
				nodes. Conversely, for the orange symbols, the 
				$P=10$ 
				nodes with highest degree were chosen as outputs and the 
				inputs were uniformly sampled from the remaining 
				$N-P=490$ 
				nodes. The black dots were obtained for the same uniform 
				random scheme as in Fig.~\ref{fig:Fig6}, were both inputs and 
				outputs were sampled uniformly~(scenario SC~\Romannum{1}).}
			\label{fig:Fig8}
		\end{figure}

		To further explore the role of hubs in networks with a more symmetric 
		assignment of in- and output nodes, we modified the uniform random 
		scheme. Instead of uniform sampling (see \ref{subsec:unif}), we now 
		ranked all the state nodes according to their degree and selected the 
		$M$ 
		nodes with the highest degree as input set 
		$S$. 
		The 
		$P=M$ 
		output nodes 
		$Z$ 
		were again uniformly sampled from the remaining nodes. As can be seen 
		from Fig.~\ref{fig:Fig8}, this preferred selection of hub nodes as 
		inputs can drastically increase the probability of invertibility in 
		scale free networks. A less drastic improvement can also be observed, 
		when the high degree nodes are used as outputs and then the input 
		nodes are uniformly sampled. 

	\subsection{Invertibility of real networks under uniform sampling of inputs and outputs (scenario~SC~\Romannum{1})}\label{subsec:real1}
		\begin{figure*}
			\centering
			\includegraphics[width = 2\columnwidth]{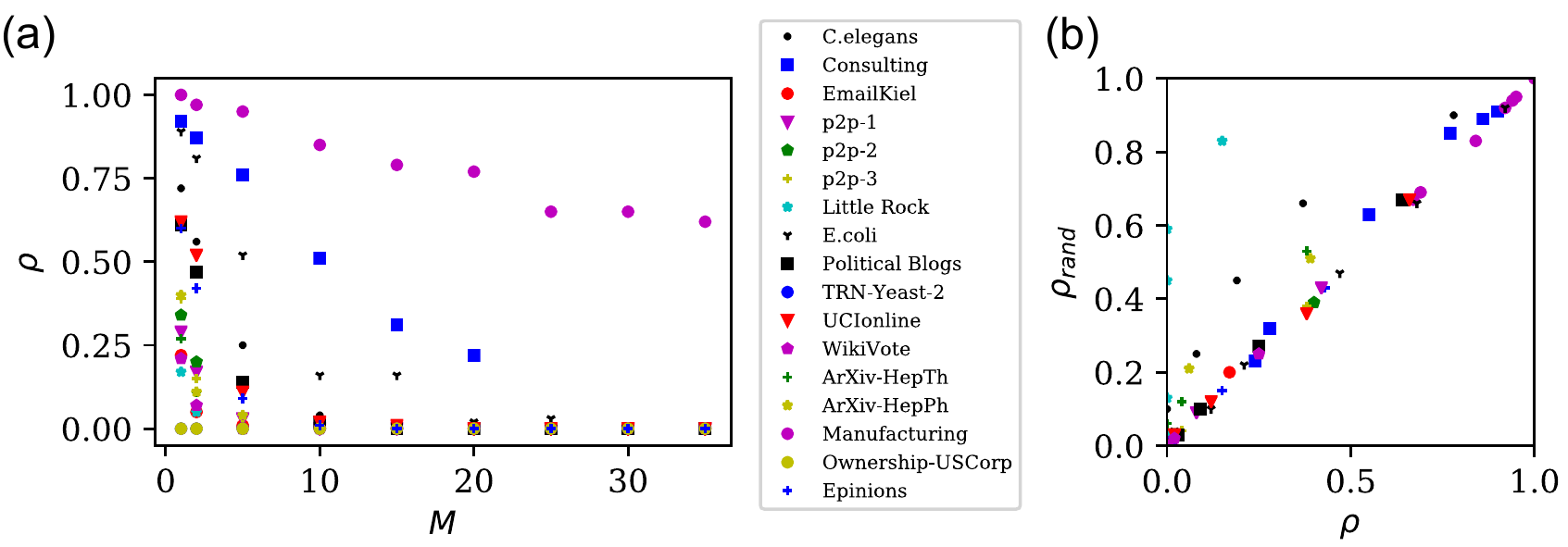}
			\caption{
				Structural invertibility of real networks. (a) The 
				probability of invertibility under the random scheme. 	
				For each network in Table~\ref{tab:Real_Networks}, 
				the disjoint sets of 
				$M=P$ 
				in- and ouput nodes were obtained by uniform random sampling 
				(scenario SC~\Romannum{1}). This was repeated 
				$100$ 
				times for each 
				$M$ 
				and the fraction 
				$\rho$ 
				of structural invertible graphs was estimated. (b) Comparison 
				of 
				$\rho$ 
				from (a) to the probability of invertibility after a 
				degree-preserving randomization 
				($\rho_\text{rand}$). Note that some of the symbols are exactly
				masking each other.
				In this randomisation, the in- and out-degree of each node 
				is preserved and the nodes linked to each other are randomly 
				selected
				\cite{
					maslov_specificity_2002}.
				}
			\label{fig:Fig9}
		\end{figure*}	
		\begin{table*}[t]
			\centering
			\resizebox{2\columnwidth}{!}{%
			\begin{tabular}{llllllll}
			\hline
			\hline
			Name& $N$ & $\vert \mathcal{E}\vert$ & $\bar{d}$ & Scalefree & 
			$\gamma$ & Brief Description & Database 
			\\
			\hline
			\multicolumn{8}{|l|}{\textit{Regulatory}} 
			\\
			\hline
			TRN-Yeast-2 
		\cite{milo_network_2002} 
			& 688 & 1079 & 3.14 & True 
			& 2.29 & Transcriptional regulatory network of S.cerevisiae & Uri 
			Alon Lab 
		\cite{noauthor_collection_nodate} 
			\\
			Ownership-USCorp 
		\cite{norlen_eva:_2002} 
			& 7253 & 6726 & 1.85 &
			True & 2.45 & Ownership network of US corporations & Pajek 
		\cite{batagelj_pajek_nodate} 
			\\
			\hline
			\multicolumn{8}{|l|}{\textit{Trust}} 
			\\
			\hline
			WikiVote 
		\cite{leskovec_community_2009} 
			& 7115 & 103689 & 29.13 & False & & Who-vote-whom network of 
			Wikipedia users & snap Stanford 
		\cite{leskovec_snap_2014}
			\\
			Epinions 
		\cite{richardson_trust_2003} 
			& 75888 & 508837 & 13.41 & True & 1.73 & Who-trust-whom network 
			of Epinions.com users & snap Stanford 
		\cite{leskovec_snap_2014} 
			\\
			\hline
			\multicolumn{8}{|l|}{\textit{Food Web}} \\
			\hline
			Little Rock 
		\cite{martinez_artifacts_1991} 
			& 183 & 2494 & 27.26 & False & & Food Web in Little Rock lake & 
			Mount Sinai 
		\cite{noauthor_mount_nodate} 
			\\
			\hline
			\multicolumn{8}{|l|}{\textit{Metabolic}} \\
			\hline
			E.coli 
		\cite{schellenberger_bigg:_2010} 
			& 1039 & 5802 & 11.17 & True & 2.61 & Network of the metabolic 
			reactions of the E. coli bacteria & BiGG 
		\cite{noauthor_bigg_nodate}
			\\
			\hline
			\multicolumn{8}{|l|}{\textit{Neuronal}} \\
			\hline
			C.elegans 
		\cite{watts_collective_1998} & 297 & 2345 & 15.79 & True & 2.15 & 
			Neural network of C.elegans & Network Repository 	
		\cite{nr}
			\\
			\hline
			\multicolumn{8}{|l|}{\textit{Citation}} \\
			\hline
			ArXiv-HepTh 
		\cite{leskovec_graphs_2005}
			& 27770 & 352807 & 25.41 & False & & Citation networks in HEP-TH 
			category of Arxiv & snap Stanford
		\cite{leskovec_snap_2014}
			\\
			ArXiv-HepPh 
		\cite{leskovec_graphs_2005} 
			& 34546 & 421578 & 24.41 & False & &Citation networks in HEP-PH 
			category of Arxiv & snap Stanford 
		\cite{leskovec_snap_2014}
			\\
			\hline
			\multicolumn{8}{|l|}{\textit{WWW}} \\
			\hline
			Political blogs 
		\cite{adamic_political_2005} 
			& 1224 & 19025 & 31.09 & True & 1.04 & Hyperlinks between weblogs 
			on US politics & Moreno 
		\cite{noauthor_datasets_nodate}
			\\
			\hline
			\multicolumn{8}{|l|}{\textit{Internet}} \\
			\hline
			p2p-1 
		\cite{leskovec_graph_2007} 
			& 10876 & 39994 & 7.36 & False & & Gnutella peer-to-peer file 
			sharing network & snap Stanford 
		\cite{leskovec_snap_2014} 
			\\
			p2p-2
		\cite{leskovec_graph_2007} 
			& 8846 & 31839 & 7.2 & False & & Same as above (at 
			different time) & snap Stanford 
		\cite{leskovec_snap_2014} 
			\\
			p2p-3 
		\cite{leskovec_graph_2007} 
			& 8717 & 31525 & 7.2 & False& & Same as above (at different time) 
			& snap Stanford 
		\cite{leskovec_snap_2014} 
			\\
			\hline
			\multicolumn{8}{|l|}{\textit{Social Communication}} \\
			\hline
			UCIonline 
		\cite{opsahl_clustering_2009} 
			& 1899 & 20296 & 21.38 & True & 1.33 & Online message network of 
			students at UC, Irvine & Opsahl 
		\cite{noauthor_datasets_2009} 
			\\
			EmailKiel
		\cite{ebel_scale-free_2002}
		 	& 57194 & 103731 & 3.63 & True & 1.77 & Email network of traffic 	
		 	data collected at University of Kiel, Germany & Barabasi 
		 \cite{liu_observability_2013} 
		 	\\
			\hline
			\multicolumn{8}{|l|}{\textit{Intraorganizational}} \\
			\hline
			Manufacturing 
		\cite{cross_hidden_2004} 
			& 77 & 2228 & 3.14 & False & & Social network from a 
			manufacturing company & Opsahl 
		\cite{noauthor_datasets_2009}
			\\
			Consulting 
		\cite{cross_hidden_2004}
			& 46 & 879 & 38.22 & False & & Social network from a consulting 
			company & Opsahl 
		\cite{noauthor_datasets_2009}
			\\
			\hline
			\hline
		\end{tabular}%
		}
		\caption{
			A compilation of networks from various fields, also examined by 
			other authors
			(\cite{
				liu_observability_2013}).
			Here 
			$N$ 
			is the number of nodes and 
			$\vert \mathcal{E} \vert$ 
			the number of edges. The column ``Scale free'' indicates whether 
			the degree distribution shows a power law and if so, the 
			power law exponent 
			$\gamma$ 
			was computed.
			}
			\label{tab:Real_Networks}
		\end{table*}

		In addition to the \textit{E.coli} metabolic network, we tested a 
		compilation of real networks (Table~\ref{tab:Real_Networks}) for 
		invertibility under the uniform input-output sampling scheme 
		(scenario~SC~\Romannum{1}). Again, we  estimated the 
		probability of invertibility 
		$\rho$ 
		as a function of the number of in- and output nodes 
		$M=P$, 
		see Fig.~\ref{fig:Fig9}(a). Here, we observe a ranking with the 
		(not scale-free) \textit{Intraorganizational} networks on top, with 
		the highest chance for structural invertibility, followed by the 
		(scale-free) biological \textit{E.coli} and \textit{C.elegans} 
		networks. Many of the remaining networks are much larger, with 
		$N > 10^5$ 
		nodes, and have a higher average degree. Nevertheless, the chance to 
		find a structural invertible in- and output configuration under the 
		uniform sampling scheme is vanishingly small already for 
		$M=P \approx 5$ 
		for many real networks in this compilation. Thus, in these networks, 
		it is typically difficult to reconstruct unknown inputs or model 
		errors, if the outputs nodes are chosen randomly and the inputs can 
		not be selected. These results are robust under a  degree-preserving 
		randomisation 
		\cite{
			maslov_specificity_2002},
		where the nodes linked to each other are randomly selected, but the 
		in-degree 
		$d_\text{in}$ 
		and out-degree 
		$d_\text{out}$ 
		of each node is preserved~(Fig.~\ref{fig:Fig9}(b)). 	

		To summarise, we find that invertibility under the uniform random 
		scheme (scenario~SC~\Romannum{1}) depends mainly on the joint distribution 
		of in- and out-degree. Dense and homogeneous networks tend to be 
		invertible, while sparse and scale-free networks provide a smaller 
		chance to reconstruct structural model errors and hidden inputs. As 
		emphasised by the results for real networks, more efficient ways to 
		select sensor nodes or inputs are required. The positive effects of 
		the preferential selection of hubs, either as inputs or outputs, hint 
		at possible ways to improve the chance for invertibility under 
		different scenarios, where only outputs (SC~\Romannum{2}) or both 
		input and output (SC~\Romannum{3}) nodes can deliberately be 
		selected.


\section{Sensor node placement for invertibility}\label{sec:sensor}

	Whilst the uniform random scheme (scenario~SC~\Romannum{1}) provides some 
	insights into the effects of network properties on invertibility, it is 
	not a very efficient strategy to randomly place the sensor nodes over the 
	network. 

	A second, more realistic scenario (SC~\Romannum{2}) is the following: 
	Assume, we have observed a systematic discrepancy between the output 
	measurements and the model and we want to infer the unknown inputs (or 
	model errors). Assume further, that the input node
	set 
	$S$ 
	is given, either from domain knowledge about the respective system or 
	from educated guessing about possible positions for input signals or 
	model errors. However, the system might not be invertible with the 
	current output node set. Typically, we know which states could in 
	principle be measured and we can define a maximum set 
	$Z_0$ 
	of potential sensor nodes. If the resulting system with the maximum 
	output set 
	$Z_0$ 
	is invertible, one can start the acquisition of time series data and feed 
	them into one of the algorithms 
	\cite{  kuhl_real-time_2011,
			schelker_comprehensive_2012,
			Fonod_boblin_unknown_2014, 
			engelhardt_learning_2016, 
			engelhardt_bayesian_2017, 
			chakrabarty_state_2017, 
			tsiantis_optimality_2018}.
	to infer the input. This approach, though straightforward, would 
	potentially be wasteful or even impractical. In domains like biology or 
	economics, measurements might in principle be possible, but costly or 
	take a great deal of time. Thus, a more feasible approach is to reduce 
	this excessive effort by selecting a minimum set of sensor nodes from the 
	maximum set 
	$Z_0$. 

	A similar sensor node placement problem for state 
	observability 
	\cite{
		boukhobza_state_2007, 
		liu_observability_2013}
	has been investigated before. In this section, we present a very simple 
	but efficient greedy algorithm to select a minimum set of sensor nodes 
	for invertibility for a given fixed set of input nodes 
	(scenario SC~\Romannum{2}). This algorithm can drastically improve the 
	chances for invertibility, as we will demonstrate by comparing to uniform 
	random sampling for scenario SC~\Romannum{1}. Finally, we will also 
	investigate a third scenario SC~\Romannum{3}, where the input nodes can 
	also be selected. 

	\subsection{Sensor node placement algorithm}

		Let us formalise the scenario SC~\Romannum{2} motivated above: The 
		influence graph 
		$g$ 
		for dynamic system~\eqref{eq:ss_model_lin} including a set 
		$S$ 
		of 
		$M$ 
		potential input nodes is assumed to be given. In addition, we have an 
		initial maximum set 
		\begin{equation}
			Z_0=\{z_1,\ldots,z_{P_0}\} \subseteq \{x_1,\ldots,x_N\}
		\end{equation}
		of 
		$P_0$ 
		potential output or sensor nodes. Thus, 
		$Z_0$ 
		incorporates all systems states which could potentially be monitored. 
		If the system with 
		$S$ 
		as given input set and 
		$Z_0$ 
		as maximum output set is not invertible, then there is no way to 
		reconstruct the inputs from the outputs. However, if invertibility is 
		ensured for 
		$Z_0$, then we want to reduce this maximum set to a smaller, potentially 
		minimal subset 
		$Z^*\subseteq Z_0$ 
		with 
		$P^*$ 
		outputs, which is still invertible, given the inputs 
		$S$. 

		From the structural invertibility theorem~(see 
		Subsec.~\ref{subsec:graphinv}) we know that the smallest output 
		set has  at least as many nodes as the input set. Thus, we will 
		always have 
		$P^* \ge M$. 
		For small sets 
		$Z_0$, 
		it might be possible to try all 
		$\binom{P_0}{M}$ 
		possible subsets of 
		$M$ 
		nodes from the maximum set 
		$Z_0$. 
		However, this brute force strategy becomes quickly infeasible, if 
		$P_0$ 
		is large. Reducing 
		$P_0$ 
		from the beginning is usually also not an option, since the maximum 
		sensor node set 
		$Z_0$ 
		needs to provide an invertible system, which might not be the case 
		for small sets. 

		A practical solution is given by a simple greedy algorithm, which 
		assumes that the triple 
		$(S,g,Z_0)$ 
		containing the maximum node set 
		$Z_0$ 
		is invertible. To initialise the algorithm, we assume that the nodes are in 
		some random order in 
		$Z_0$. 
		In the first iteration, we select the first node 
		$z$
		from 
		$Z_0$
		and try to delete it, but only if 
		$(S,g,Z_0 \setminus z)$ 
		with 
		$P_0-1$ 
		sensor nodes is still invertible. If not, we keep 
		$z$ 
		in the node set 
		$Z_1 := Z_0$, 
		i.e. we reject the deletion of this node. Otherwise we delete the 
		node by setting 
		$Z_1:=Z_0\setminus z$.
		In any case, we continue and try to delete a different node, say 
		the next node in 
		$Z_1$. 
		We proceed in this way until we have a sensor node set 
		$Z_k$ with $P_k = M$ output nodes and set $Z^* = Z_k$. 
		This algorithm takes at most 
		$P_0-M$ 
		steps. Note, that the greedy algorithm will always find a minimum 
		node set with the minimum number 
		$P^*=M$ 
		of sensor nodes, provided 
		$(S,g,Z_0)$ 
		with the initial node set 
		$Z_0$ 
		is invertible.   

		In Fig.~\ref{fig:Fig10} we present an example for a network with 
		$N=1000$ 
		nodes and 
		$M=100$ 
		uniformly sampled input nodes. All other nodes were included in the 
		maximum output set 
		$Z_0$, i.e.~$P_0=900$. 
		The algorithm takes $598$ iterations to find a minimum node set 
		$Z^*$ 
		with 
		$P^*=100$.~The 
		number of iterations can be reduced by replacing the random 
		removal of output nodes by a more selective satellite-deletion 
		strategy. By ranking the nodes in 
		$Z_0$ 
		according to their degree and selectively removing nodes with low 
		degree shrinks the (invertible) output set much faster than random 
		deletion. Reversing the order of the ranking, i.e. trying to 
		selectively delete hubs from the set of sensor nodes results in more 
		rejections and thus more iterations (Fig.~\ref{fig:Fig10}).~This is 
		consistent with the results from (Fig.~\ref{fig:Fig8}), were we found 
		that preferential selection of high degree output nodes improves the 
		chance for invertibility. Please note, that the greedy algorithm is 
		usually fast enough for most purposes, even without ranking the 
		nodes. This analyses merely serves to better understand the role of 
		hubs as inputs or outputs. 
		\begin{figure}
			\centering
			\includegraphics[width=\columnwidth]{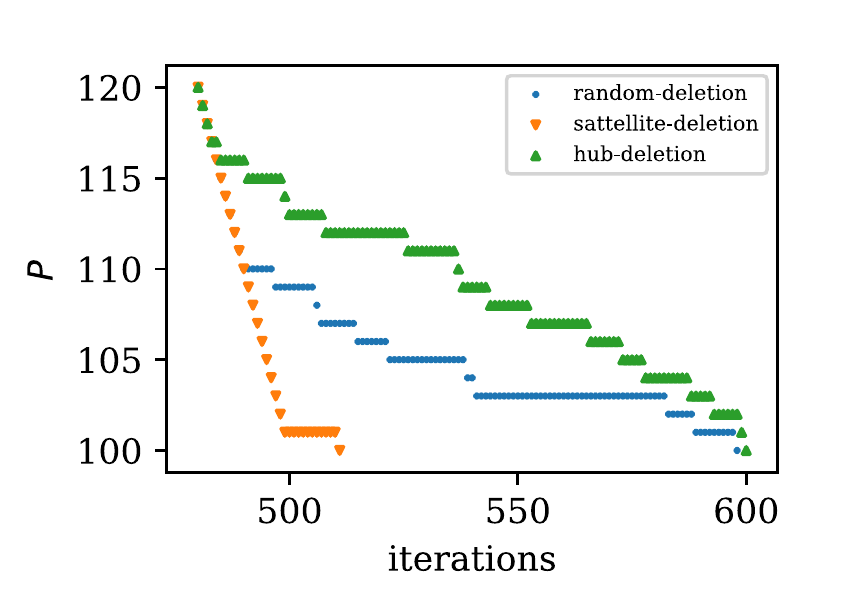}
			\caption{
				The greedy algorithm for optimal sensor node selection for 
				invertibility applied to a network with 
				$N=1000$ 
				state nodes, and 
				$M=100$ 
				randomly distributed inputs. From the remaining 
				$900$ 
				nodes, the maximum output set 
				$Z_0$ with 
				$P_0=700$ 
				potential sensor nodes was randomly chosen. The plot compares 
				the number of sensor nodes 
				$P$ 
				versus the number of iterations of the greedy algorithm for 
				random-, satellite-, and hub-deletion strategies. 
				}
			\label{fig:Fig10}
		\end{figure}

	\subsection{Application to real networks}
		The sensor node placement algorithm can only be successful, if the 
		maximum sensor node set 
		$Z_0$ 
		(together with the given input node 
		set) yields an invertible system. Larger 
		$Z_0$, 
		i.e.~a larger number 
		$P_0$ 
		of potentially measurable outputs, will obviously increase the 
		chances for invertibility and thus also the chance to find a minimum 
		sensor node set 
		$Z^*$ 
		of cardinality 
		$P^*=M$. 
		Apart from directly measuring all possible nodes, the largest 
		possible nontrivial sensor node set 
		$Z_0$ 
		is given by the 
		$N-M$ 
		nodes which are not input nodes. We used this 
		$Z_0$ 
		with 
		$P_0=N-M$ 
		to check the compendium of real networks 
		(Table~\ref{tab:Real_Networks}) for invertibility. 
		\begin{figure*}
			\centering
			\includegraphics[width = 2\columnwidth]{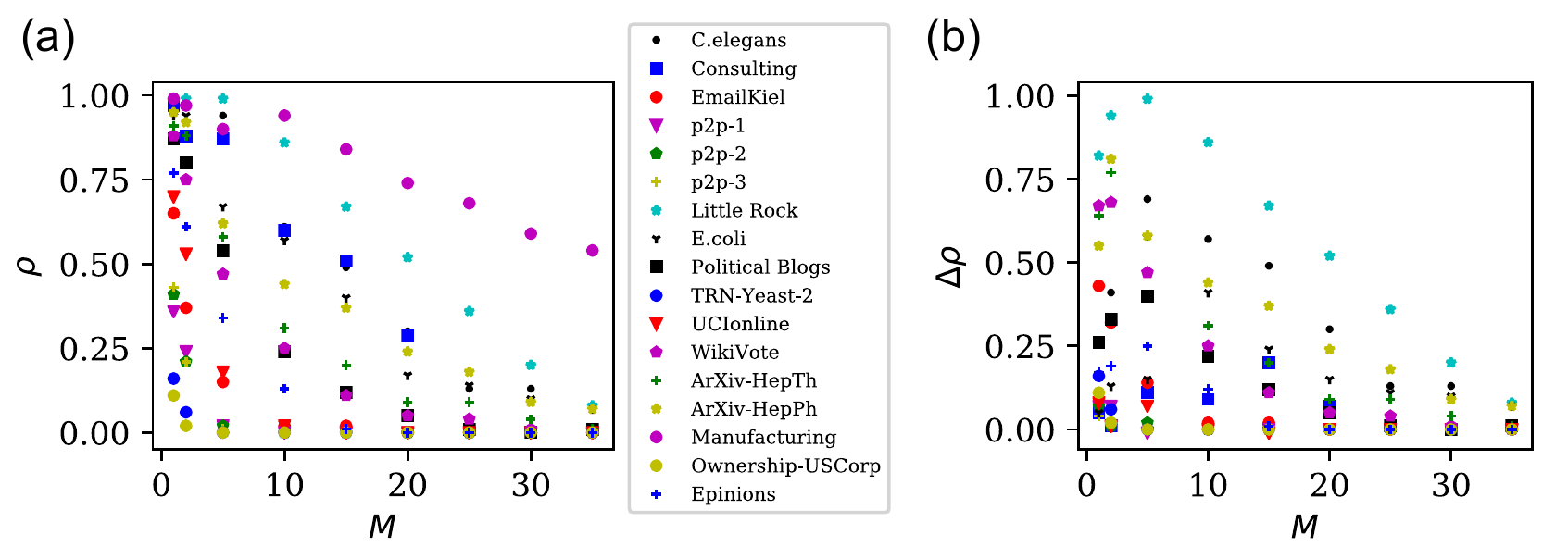}
			\caption{	
				The sensor node placement algorithm can increase the 
				probability of invertibility 
				$\rho$, 
				if the input node set is given (scenario SC~\Romannum{2}). 
				(a) The input node set was again uniformly sampled. All other 
				nodes were considered as potential maximum output node set 
				$Z_0$. 
				In case of invertibility, the sensor node placement 
				algorithm can reduce 
				$Z_0$ 
				to a minimum sensor node set 
				$Z^*$ 
				with 
				$P^*=M$ 
				outputs. (b) The difference 
				$\Delta \rho$ 
				between the probability of invertibility 
				$\rho$ 
				in (a) and the uniform random selection scheme used in 
				Fig.~\ref{fig:Fig9}(a) to highlight the improvement.
				}
			\label{fig:Fig11}
		\end{figure*}

		For the results in Fig.~\ref{fig:Fig11}(a), we sampled 
		$M$ 
		input nodes 
		$S$ 
		uniformly from all nodes  and then selected 
		$Z_0$ 
		as the remaining 
		$N-M$ 
		nodes~(scenario SC~\Romannum{1} with the largest possible 
		$Z_0$). 
		We repeated this over 100 randomly sampled input sets 
		$S$ 
		and estimated the fraction 
		$\rho$ 
		of invertible systems $(S,g,Z_0)$. Thus,
		$\rho$ provides an estimate for the probability to obtain 
		a structurally invertible system (for a given graph $g$) under scenario SC~\Romannum{2},
		where the $M$ input nodes $S$ are given and cannot be chosen, but
		all the other nodes can in principle be measured. This is then
		identical to the fraction of systems, where the sensor node 
		algorithm can reduce this initial sensor node set $Z_0$ to 
		a minimum set $Z^*$ with $P^*=M$ outputs. By comparing Fig.~\ref{fig:Fig11}(a) with 
		Fig.~\ref{fig:Fig9}(a) we can observe, that this strategy improves 
		(in some cases drastically)  the chances to find an invertible 
		system. For better visibility see also Fig.~\ref{fig:Fig11}(b), where 
		we have plotted the difference 
		$\Delta \rho$ 
		between optimal sensor node placement in Fig.~\ref{fig:Fig11}(a) and 
		uniform output sampling Fig.~\ref{fig:Fig9}(a). 

	\subsection{Input node selection}
		So far we have assumed the input node set 
		$S$ 
		as given. To mimick this frequent situation that the input nodes can 
		not deliberately be selected, we performed uniform random sampling of 
		the input nodes. However, there might be situations
		where we can influence the selection of inputs 
		(scenario SC~\Romannum{3}). For the design of communication networks, 
		the ability to uniquely distinguish different input signals is 
		clearly a requirement and input node  sets are often deliberately 
		chosen~\cite{larsson_design_2014}. Another example is given by the modular approach to model 
		building, where one aims to describe a subsystem (or module) by a 
		system of ODEs~\cite{raue_lessons_2013, almog_is_2016}. This module will by definition receive 
		inputs from the environment (see Fig.~\ref{fig:Fig1}), which might not be 
		directly measurable. Then, invertibility to infer these inputs from 
		outputs is clearly an important requirement, which might influence 
		decisions about the right state variables to include in the 
		subsystem.

		Based on the results of Fig.~\ref{fig:Fig8} we hypothesised that hubs 
		with a high node degree are good candidates for inputs promoting 
		invertibility. To test this hypothesis, we used again the networks 
		listed in Table~\ref{tab:Real_Networks}. For each network 
		$g$, 
		we selected the 
		$M$ 
		nodes with highest out-degree as input node set 
		$S$. 
		As before, we used the remaining 
		$N-M$ 
		nodes as maximum sensor node set 
		$Z_0$. 
		If 
		$(S,g,Z_0)$ 
		is invertible, the sensor node selection algorithm can always reduce 
		this to 
		$(S,g,Z^*)$ 
		with a minimum node set 
		$Z^*$ 
		having only 
		$P^*=M$ 
		outputs. Starting from 
		$M=1$ 
		we increased the number of input nodes for each network one by one as 
		long as the corresponding system 
		$(S,g,Z_0)$ 
		was invertible. The maximum number 
		$M_\text{max}$ 
		of inputs for each network  is shown in Fig.~\ref{fig:Fig12}, 
		together with the probability of invertibility for the other 
		scenarios using the uniform random scheme (data from 
		Fig.~\ref{fig:Fig9}(a)) and the sensor node placement algorithm~(data 
		from Fig.~\ref{fig:Fig11}(a)). Please note, that for this hub input 
		selection strategy all systems with 
		$M<M_\text{max}$ 
		inputs are invertible. Clearly, the hub input selection strategy 
		provides a way to increase the number of input signals which can 
		still reconstructed in these networks.
		\begin{figure*}
			\centering
			\includegraphics[width=2\columnwidth]{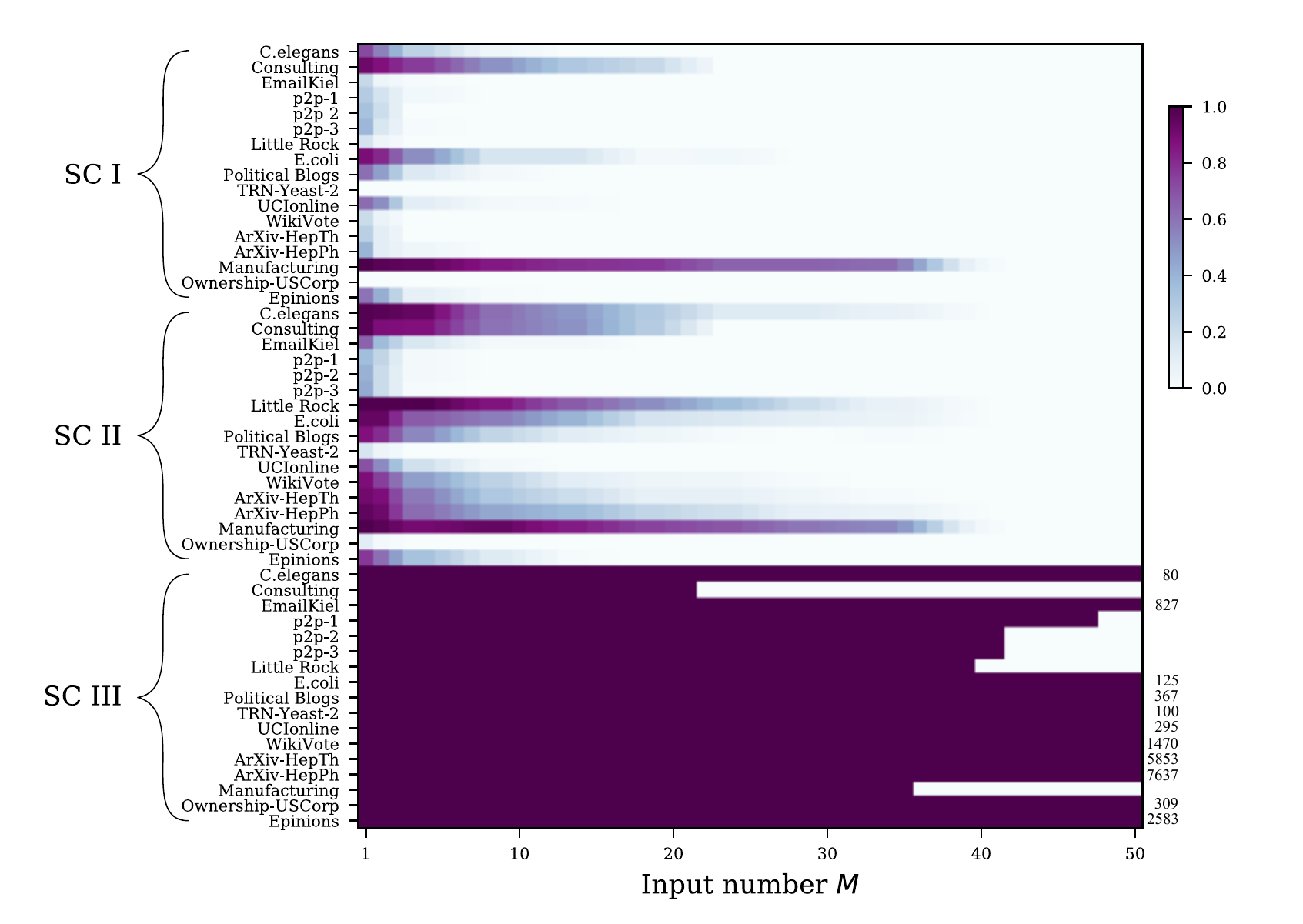}
			\caption{
				A heatmap comparing the probability of invertibility 
				$\rho$ 
				for the networks in Table~\ref{tab:Real_Networks} with three 
				different node selection schemes as a function of the number 
				of inputs. Scenario~SC~\Romannum{1}, where neither inputs nor 
				outputs can be selected, was simulated by uniform random 
				sampling of input and output nodes  (same data as in 
				Fig.~\ref{fig:Fig9}(a)). In scenario~SC~\Romannum{2} we are 
				able to select the output nodes and the results for the 
				optimal sensor node placement from Fig.~\ref{fig:Fig11}(a) 
				are included for comparison. 
				Under scenario scenario~SC~\Romannum{3}, we can select both, 
				input and output nodes. We used hub input selection (selection of nodes
				with highest out-degree) in 
				combination with optimal sensor node placement and found that 
				$\rho$ 
				is either one or zero. The numbers on the right indicate the 
				maximum number 
				$M_\text{max}$ 
				of inputs which still provide an invertible system. 
				}		
			\label{fig:Fig12}
		\end{figure*}

\section{Summary and open questions}\label{sec:conclusion}
\subsection{Summary and significance of the results}
		Reconstructing unknown inputs from outputs of open systems is useful in 
	many settings. For modellers, the inputs provide important information 
	about model errors and cues for model improvement or extension
	\cite{
		mook_minimum_1987, 
		engelhardt_learning_2016, 
		engelhardt_bayesian_2017}.
	In biomedical systems, the unknown inputs can represent unmodelled 
	environmental or physiological inputs, which might be interesting for the 
	design of devices or measurement strategies.~In electrical or secure 
	networks, the unknown inputs could be attack signals, which need to be 
	reconstructed and then mitigated. Unknown inputs can also be useful for 
	improved state estimation 
	\cite{
		engelhardt_learning_2016, 
		engelhardt_bayesian_2017, 
		chakrabarty_state_2017}
	and data assimilation 
	\cite{
		abarbanel_predicting_2013, 
		reich_probabilistic_2015}.
	Thus, from the viewpoint of modellers and engineers, invertibility is a 
	desirable property for open systems. In this work, we have 
	focused on structural invertibility, which has the two advantages of (i) 
	only requiring topological network information and  (ii) being testable 
	even for large networks using the structural controllability algorithm. 
	
	Although invertibility is desirable from an applied and analytic 
	perspective, our results for an uniform random input selection scheme 
	indicate that invertibility cannot be taken for granted, especially not 
	in networks with low average degree, many inputs and with a scale free 
	degree distribution. Thus, under the scenario SC~\Romannum{1}, were 
	neither inputs nor outputs can deliberately be chosen, many real networks 
	have a disposition to mask differences between different input signals. 
	It is well known that for example some dynamic biological systems respond 
	often identically or similarly to a variety of different stimuli. Thus, 
	living dynamic systems often  distinguish different patterns rather than 
	small differences in the inputs. In addition, real systems need to have a 
	certain robustness against small perturbations and noise. Thus, 
	invertibility is possibly not always a desirable property for the 
	specific tasks to be performed by the network. Further research is needed 
	to investigate tradeoffs between invertibility and other network traits, 
	like e.g. controllability or robustness.  Nevertheless, non-invertibility 
	poses a challenge for experimentalists and modellers to reconstruct 
	structural model errors and inputs from the environment. 

	We approached this problem by deriving an efficient sensor node placement 
	algorithm, which extracts a minimum set of measurement nodes required for 
	invertibility of a given network with a given input set (scenario 
	SC~\Romannum{2}). In this scenario, the sensor node placement algorithm 
	facilitates optimal experimental design for the reconstruction of inputs 
	from outputs. As such, it can be used in conjunction with input 
	reconstruction algorithms
	\cite{kuhl_real-time_2011,
		schelker_comprehensive_2012, 
		Fonod_boblin_unknown_2014,
		engelhardt_learning_2016, 
		engelhardt_bayesian_2017}
	and input observers
	\cite{
		chakrabarty_state_2017, 
		tsiantis_optimality_2018}.
	Structural invertibility provides a necessary condition for these algorithms to work.
		
	In a third scenario SC~\Romannum{3} we assumed that both inputs and 
	outputs can be selected. We found that selecting nodes with a high 
	out-degree as input nodes, in combination with optimal output selection using 
	the sensor node placement algorithm, drastically increases the number of 
	inputs which can still be reconstructed from output measurements. 
	Intuitively, these input hubs distribute the input signal widely over 
	the network, therefore increasing the likelihood for finding node 
	disjoint paths linking these inputs to the outputs. Although scenario 
	SC~\Romannum{3}, where input nodes can deliberately be selected, might 
	not always be realistic, it can be useful for the design of dynamic 
	mathematical models or for the design of synthetic systems. For example, 
	a key goal of Synthetic Biology is to engineer new biological systems for 
	desired functionalities. In general, these systems are embedded in larger 
	systems and will receive inputs from their environment, which should be 
	inferable from measurements. In this case, optimal input selection in 
	conjunction with optimal sensor node placement can provide important 
	benefits. Another example is modular modelling, were an interesting 
	subsystem embedded in a larger system is modelled in detail~\cite{raue_lessons_2013, almog_is_2016}. To detect 
	both  potential model error or genuine inputs from the environment to the 
	model, invertibility is essential. To achieve invertibility, it might be 
	useful to include additional states which otherwise would not be deemed 
	to be essential to understand the modular subsystem. 
	
	\subsection{Limitations and open questions}
	Purely structural approaches to controllability and observability have been criticised 
	to sometimes provide suboptimal conclusions for real systems. Depending on the 
	quantitative properties of the interactions between the state nodes, a system 
	might be  practically uncontrollable or unobservable, 
	even if the structural criteria are fulfilled
	\cite{ 	krener_meas_2009,
			cornelius_controlling_2011, 
			sun_controllability_2013, 
			cornelius_realistic_2013, 
			yan_spectrum_2015,
			lo_iudice_structural_2015,
			summers_submodularity_2016, 
			klickstein_energy_2017,
			aguirre_observ_2018,
			haber_state_2018}.
	Non-binary indices quantifying the degree of 	
	controllability and observability have been devised. These indices
	require at least the algebraic structure of the coupling functions between 
	the state nodes~\cite{aguirre_observ_2018} or even the full functional form
	and the parameters of the network~\cite{cornelius_realistic_2013}.
	
	A similar caveat applies to structural invertibility, which is only a condition 
	for the existence of the inverse input-output map. As in any inverse problem~	
	\cite{nakamura_potthast_inverse_2015}, this
	might not be sufficient to actually implement this inverse map 
	for reconstructing unknown inputs (including model errors) from outputs. 
	Some unknown input signals might be hard to detect by noisy sensors with limited sensitivity. 
	As discussed in Subsec.~\ref{ssec:practinv}, these  issues are 
	related to the mathematical fact that the inverse of the compact input-output map is not continuous. 
	Devising an index or condition number for the degree of invertibility (or 
	continuity of the 
	inverse) is therefore an important question for future research. Such an invertibility index might 
	be used to rank sensor nodes, which could readily be utilised in a straightforward modification of
	our sensor 
	node placement algorithm~(Sec.~\ref{sec:sensor}). The desired invertibility index might also be 
	useful for improving or designing unknown input reconstruction algorithms
	~\cite{kuhl_real-time_2011,
		schelker_comprehensive_2012, 
		Fonod_boblin_unknown_2014,
		engelhardt_learning_2016, 
		engelhardt_bayesian_2017,
		chakrabarty_state_2017, 
		tsiantis_optimality_2018},
	which adapt their regularisation automatically to the degree of discontinuity of the 
	inverse input-output map. 
	
	A further potential objection against structural invertibility is that the influence
	graph has to be completely known. Indeed, if the aim is to assess invertibility of the
	true but unknown systems structure, we might obtain erroneous results if we use an incorrect
	influence graph with missing or spurious edges. However, the aim of systems inversion
	is to detect unknown inputs including systematic model errors. Missing or spurious 
	interactions in the incorrect model graph are therefore causes of unknown inputs.
	To reconstruct the model errors, it is important that our potentially incorrect or 
	incomplete model is invertible. 	
	
	Currently, our invertibility results are limited to deterministic systems described
	by ODEs. However, the definition of invertibility doesn't exclude stochastic
	unknown unput functions. Nonlinear extensions of the Kalman-Filter~\cite{kuhl_real-time_2011, 
	schelker_comprehensive_2012, Fonod_boblin_unknown_2014} or fully probabilistic approaches
	~\cite{engelhardt_bayesian_2017} to reconstruct the moments or even the full probability 
	distribution of the unknown input will only work for invertible systems. Their actual
	utility to simultaneously overcome the discontinuity of the inverse input-output map and 
	to estimate probabilistic features of unknown inputs should systematically be explored. 
	
	Invertibility of systems with intrinsic process noise, which are often described by stochastic 	
	differential equations (SDEs), 	is a largely unexplored field. First, a modified probabilistic 
	definition of 
	invertibility is required for systems governed by SDEs. Second, the role of the different sources
	of noise for the invertibility needs to be investigated. Recent results for stochastic 
	synchronisation~\cite{russo_noise_ind_sync2018, burbano_pinning_noise_2019} indicate, that noise can 
	have both, detrimental and beneficial effects. 
	It would be exciting to investigate, whether similar effects are possible for unknown input
	reconstruction.

	To conclude, invertibility (or the lack of it) has important implications 
	for modelling frameworks and strategies to deal with incomplete and 
	uncertain systems. Our analysis and algorithm for optimal experimental 
	design are only a first step towards more sophisticated methods 
	specifically tailored to handle systematic model errors and open
	systems. We belive that these approaches will increase our ability to better 
	understand and manipulate complex systems, even if our knowledge will not be complete.


\begin{acknowledgments}
	This paper is part of the SEEDS project, which is funded by the Deutsche Forschungsgemeinschaft (DFG project number 354645666).
\end{acknowledgments}

\clearpage

\appendix*

\section{Derivation of the Algebraic Criterion}\label{app:Proof}
	For the sake of completeness let us explicitly write down the general mathematical setting 
	for the linear case in our notation and deduce the algebraic criterion. In the 
	general linear case, the 
	initial state is given by a vector 
	$\vec{x}_0\in\mathbb{R}^N$, 
	the dynamic of the system is given by 
	$A\in\mathbb{R}^{N\times N}$, 
	and 
	$C\in\mathbb{R}^{P \times N}$ 
	maps the system state to the output. There might be a known input 
	$\vec{u}:[0,T]\to \mathbb{R}^{M'}$ 
	distributed by 
	$B\in \mathbb{R}^{N\times M'}$ 
	over the state variables. Finally
	$D \vec{w}$ 
	models the structural model errors, so that 
	we get the linear dynamic system,
	\begin{subequations}
		\begin{eqnarray}
		\dot{\vec{x}}(t)  &=& A\vec{x}(t) + B\vec{u}(t) + D\vec{w}(t) \\
		\vec{x}(0) &=& \vec{x}_0 \\
		\vec{y}(t) &=& C\vec{x}(t) \, .
		\end{eqnarray}
		 \label{eq:DynamicSystem}
	\end{subequations}	
	Let $\Phi: \vec{w} \mapsto \vec{y}$ 
	denote the solution operator, that maps the input $\vec{w}$ to the output $\vec{y}$ 
	according to the dynamic system \eqref{eq:DynamicSystem}.
	For the dynamic system \eqref{eq:DynamicSystem} we introduce the homogeneous system 
	\begin{subequations}
		\begin{eqnarray}
		\dot{\vec{x}}(t)  &=& A\vec{x}(t)+ D\vec{w}(t) \\
		\vec{x}(0) &=& 0 \\
		\vec{y}(t) &=& C \vec{x}(t)
		\end{eqnarray}
	\end{subequations}
	with the solution operator $\Phi^\text{hom}: \vec{w} \mapsto \vec{y}$.	
	Recall that a system is called invertible if for given data 
	\begin{equation}
		\vec{y}^\text{obs}:[0,T]\to \mathbb{R}^P
	\end{equation}
	any solution of
	\begin{equation}
		\Phi (\vec{w})(t) = \vec{y}^\text{obs}(t) \quad \forall \, t\in [0,T]
		\label{eq:PhiEquation}
	\end{equation}	
	is unique. We will make use of the 
	Volterra-operator 
	\begin{equation}
		V (\vec{w})(t) := \int_0^t \vec{w}(s) \, \text{d}s \, .
	\end{equation}
	The Volterra-operator has the property
	\begin{equation}
		V^n (\vec{w})(t) = \int\limits_0^t \frac{(t-s)^{n-1}}{(n-1)!} \vec{w}(s) \, \text{d}s
	\end{equation}		
	and hence
	\begin{equation}
		\begin{aligned}
		&\sum_{n=0}^\infty C A^nD V^{n+1}(\vec{w})(t) \\ &\quad = 
		\int\limits_0^t C \exp((t-s)A) D \vec{w}(s) \, \text{d}s
		\, ,
		\end{aligned}
	\end{equation}
	where the integration is understood component-wise.
	\begin{lemma} \label{lemma:PhiHom}
		Let
		$\Phi$ 
		and 
		$\Phi^\text{hom}$ 
		be the solution operators of a dynamic system as defined above. Then
		\begin{enumerate}
			\item $\Phi^\text{hom}$ is linear, continuous, and compact,
			\item $\Phi(w+v)= \Phi (w)$ if and only if $\Phi^\text{hom}(v)=0$.
		\end{enumerate}
	\end{lemma}
	\begin{proof}
	\begin{enumerate}
		\item
		The homogeneous solution operator takes the form
		\begin{equation}
			\Phi^\text{hom} = \sum_{n=0}^\infty (CA^n D) V^{n+1} \quad .
		\end{equation}
		To see this, let
		\begin{equation}
			\vec{x}(t) =  \sum_{n=0}^\infty A^nD V^{n+1}(\vec{w})(t) 
		\end{equation}
		then
		\begin{align}
			\frac{\text{d}}{\text{d}t} \vec{x} (t) &= 
			D \vec{w}(t) + A \sum_{n=1}^\infty A^{n-1}D V^{n}(\vec{w})(t) \\
			& = D\vec{w}(t) + A\vec{x}(t)
		\end{align}
		which shows that $\vec{x}$ solves the dynamic equations as well as 
		$\vec{x}(0)=0$. Multiplication with 
		$C$ yields the solution of the homogeneous system
		\begin{equation}
			\vec{y}(t) = \Phi^\text{hom}(\vec{w})(t) \quad .
		\end{equation}
		Since $V^n$ is linear, so is $\Phi^\text{hom}$. 
		As we make the restriction $w_i\in L^2([0,T])$ we get that 
		$\Phi^\text{hom}$ is Hilbert-Schmidt thus continuous and compact.
		\item
		The inhomogeneous part of $\Phi$ is given by
		\begin{equation}
		\begin{aligned}
		\phi(t) := &\sum_{n=0}^\infty CA^n D V^{n+1} (\vec{u})(t) \\ &+ \exp(At)\vec{x}_0
		\end{aligned} \, ,
		\end{equation}
		such that the full solution can be written as
		\begin{equation}
			\Phi (\vec{w})(t) = \phi(t)+ \Phi^\text{hom} (\vec{w})(t) 
		\end{equation}				
		which shows that
		\begin{equation}
			\Phi(\vec{w}+\vec{v}) = \Phi(\vec{w}) + \Phi^\text{hom} (\vec{v})
		\end{equation}
		hence if and only if 
		$\Phi^\text{hom}(\vec{v})=0$ 
		then 
		$\vec{v}$ 
		leaves the solution of the inhomogeneous system invariant.
	\end{enumerate}
	\end{proof}
	As a result of the above lemma, we can set $\vec{x}_0$ as well as $\vec{u}$ to zero. 
	Furthermore this shows that $\Phi$ given in Sec.~\ref{sec:criteria_inv} 
	is indeed 
	the relevant solution operator
	that has to be one-to-one.
	We now follow the proof of Sain and Massey
	\cite{sain_invertibility_1969}.
	After Laplace-transformation the dynamic equation becomes
	\begin{equation}
		s \mathcal{L}[\vec{x}](s) = A \mathcal{L}[\vec{x}](s) + D \mathcal{L}[\vec{w}](s)
	\end{equation}
	with a complex variable $s\in\mathbb{C}$. Using the transfer function
	\begin{equation}
		T(s):=C(\mathbb{I}s - A)^{-1}D \, ,
	\end{equation}
	where $\mathbb{I}$ is the identity in $\mathbb{R}^N$, we get
	\begin{equation}
		T(s) \mathcal{L}[\vec{w}](s) = \mathcal{L}[y](s) 
	\end{equation}
	thus $T$ is the Laplace-transform of the solution operator $\Phi$, and 
	since $\mathcal{L}[\vec{w}]$ is the zero function if and only if $\vec{w}$ is the 
	zero-function (almost everywhere), we know that
	$\Phi$ is one-to-one if and only if $T$ has $\text{rank}\,T(s)= M$ 
	for $s\in \mathbb{C}$ (almost everywhere), i.e. if from $\mathcal{L}[\vec{y}]=0$ it 
	follows, that $\mathcal{L}[\vec{w}]=0$.
	We make the assumption, that $\vec{w}$ is smooth, which is equivalent to the assumption, 
	that it can be written as a Laurent-series (comprising only the principal part)
	\begin{equation}
		\mathcal{L}[\vec{w}](s) = \frac{1}{s} \sum\limits_{k=0}^\infty \frac{1}{s^k} \xi_k
	\end{equation}
	in Laplace-space, where $(\xi_k)_k$ is a sequence of $\mathbb{R}^M$ vectors. Using the 
	Neumann-series yields
	\begin{equation}
		C(\mathbb{I}s - A)^{-1} D = \frac{1}{s} \sum\limits_{l=0}^\infty C \frac{A^l}{s^l} D 
		\, .
	\end{equation}
	If $\mathcal{L}[\vec{y}]=0$, then
	\begin{equation}
		\frac{1}{s^2} \sum\limits_{k,l =0}^\infty \frac{1}{s^{l+k}} CA^l D \xi_k = 0	
	\end{equation}
	and since $\{1,s^{-1},s^{-2},\ldots\}$ are linearly independent in function space, 
	by equating coefficients for each $n\in\mathbb{N}_0$ we find
	\begin{equation}
		\sum\limits_{k=0}^n CA^k D \xi_{n-k} = 0 \, .
	\end{equation}		
	It is now convenient to define
	\begin{equation}
		R_n := \begin{bmatrix}
		CD & CAD & \ldots & CA^nD
		\end{bmatrix}
	\end{equation}
	as known from the Kalman controllability matrix, as well as
	\begin{equation}
		\Xi_n := \begin{bmatrix}
			\xi_n \\ \vdots \\ \xi_0
		\end{bmatrix}
	\end{equation}
	to finally get
	\begin{equation}
		R_n \Xi_n = 0 \, \forall n \in \mathbb{N}_0 \label{eq:MXi} \, .
	\end{equation}
	Thus, the dynamic system is invertible if and only if we find a sequence $(\xi_k)_k$ 
	such that \eqref{eq:MXi} holds.
	If we combine $R_0,R_1,\ldots , R_l$ to one matrix
	\begin{equation}
			Q_l := \begin{bmatrix}
				CD & CAD & \dots & CA^lD \\ 
				0  & CD  & \dots & CA^{l-1}D \\
				\vdots&    &   \ddots   & \vdots  \\
				0     &  \dots &     &  CD
			\end{bmatrix} \quad .
	\end{equation}	
	we find, that \eqref{eq:MXi} is equivalent to
	\begin{equation}
			\text{rank}\,Q_{N-1} - \text{rank}\, Q_{N-2} = M \, , 
			\label{eq:Sain}
	\end{equation}
	the criterion stated by Sain and Massey.
	From \eqref{eq:MXi} we directly see, that $\Xi_l \in \ker R_l$ and $\Xi_{l+1} 
	\in \ker R_{l+1}$.
	Also $\Xi_{l+1} = [\xi_{l+1} , \Xi_l]^T$, thus
	\begin{equation}
		\Xi_{l+1} \in \ker R_{l+1} \cap \left( \mathbb{R}^M \times \ker R_l \right) \, .
	\end{equation}
	If we now exclude the trivial solution $\xi_l = 0$ for all $l$, this motivates
	\begin{equation}
		K_0 := \ker R_0 \backslash \{0\}
	\end{equation}
	and
	\begin{equation}
		K_l :=  \ker R_{l+1} \cap \left( \mathbb{R}^M \times K_l \right) \, .
	\end{equation}
	As we iterate though $K_0,K_1,\ldots$, as long as $K_l \neq \emptyset$ there is a 
	non-trivial solution of 
	\begin{equation}
		R_k \Xi_k = 0 \text{  for } k \leq l \, .
	\end{equation}
	Hence, if and only if we find a $l\in\mathbb{N}_0$, such that $R_l = \emptyset$, then 
	the dynamic system is invertible. From \eqref{eq:Sain} one can see, that is suffices 
	to check only $l \leq N-1$.
	In addition to that we find the following theorem.
	\begin{theorem}\label{lemma:xiShift}
		The solution space of
		\begin{equation}
			\Phi(\vec{w}) = \vec{y}^\text{obs} 
		\end{equation}			
		is either zero- or infinite-dimensional.
	\end{theorem}
	\begin{proof}
		From the considerations above it is clear, that we have to show that the 
		solution space of \eqref{eq:MXi} is either zero- or infinite-dimensional, i.e. 
		if there is a non-trivial sequence $(\xi_k)_k$, then there is an infinite-dimensional 
		vector space of such sequences. 
		
		First, assume $(\xi_k)_k$ is a sequence, that solves
		\eqref{eq:MXi} and $\xi_0 = 0$. We define another sequence $(\chi_k)_k$ by 
		$\chi_k := \xi_{k+1}$, i.e. $(\chi_k)_k$ is the left shift of $(\xi_k)_k$. 
		Let $X_l$ the 
		vector $[\chi_l,\ldots , \chi_0]^T$ analogous to $\Xi_l$.
		Then 
		\begin{equation}
			0= R_l \Xi_l = R_{l-1} X_{l-1}
		\end{equation}
		for all $l$. Hence $(\chi_k)_k$ is also a solution.		
		This shows, that if $(\xi_k)_k$ solves \eqref{eq:MXi}, then 
		we can without loss of generality assume $\xi_0 \neq 0$.		
		
		Now let $(\xi_k)_k$ a solution and define $\chi_0 := 0$ and $\chi_k=\xi_{k-1}$, i.e. 
		$(\chi_k)_k$ is the right shift of $(\xi_k)_k$. Then
		\begin{equation}
			R_0 X_0 = 0
		\end{equation}
		is clear, and for $l\geq 1$
		\begin{equation}
			R_l X_l = R_{l-1} \Xi_{l-1} = 0
		\end{equation}
		hence the sequence $(\chi_k)_k$ solves \eqref{eq:MXi}. Let henceforth 
		$\mathfrak{S}_\rightarrow$ denote 
		the right shift.
		
		Since matrix multiplication is a linear operation it is clear, that if $(\xi_k)_k$ 
		and $(\chi_k)_k$ solve \eqref{eq:MXi}, so does $(\xi_k+\chi_k)_k$ as well as 
		$(a \xi_k)_k$ for an arbitrary real number $a$. Therefore the space of sequences that 
		solve \eqref{eq:MXi} is a real vector space, denoted $\mathcal{K}$. Let $(\xi_k)_k 
		\in \mathcal{K}$ with $\xi_0 \neq 0$. Then 
		$\mathfrak{S}_\rightarrow^q (\xi_k)_k \in \mathcal{K}$ and 
		\begin{equation}
			\{ (\xi_k)_k, \mathfrak{S}_\rightarrow (\xi_k)_k, 
			\mathfrak{S}_\rightarrow^2 (\xi_k)_k , \ldots \}
		\end{equation}
		is a set of infinitely many linearly independent vectors in $\mathcal{K}$, hence
		\begin{equation}
			\dim \mathcal{K} = \infty \, .
		\end{equation}		
	\end{proof}


%


\end{document}